\documentclass{article}

\usepackage{amsmath}
\usepackage{amsfonts}
\usepackage{amssymb}
\usepackage{color,graphicx}
\usepackage{epstopdf}
\usepackage{graphics}
\usepackage{subfigure}

\graphicspath{{./figs/}{./res_circle/},{./res_submarine/}}

\newtheorem{theorem}{Theorem}

\newtheorem{corollary}[theorem]{Corollary}

\newtheorem{definition}[theorem]{Definition}

\newtheorem{lemma}[theorem]{Lemma}

\newtheorem{remark}[theorem]{Remark}

\newenvironment{proof}[1][Proof]{\noindent\textbf{#1.} }{\ \rule{0.5em}{0.5em}}

\newcommand*{\mcI}{\mathcal{I}}
\newcommand*{\mcK}{\mathcal{K}}
\newcommand*{\mcT}{\mathcal{T}}
\newcommand*{\mcA}{\mathcal{A}}
\newcommand*{\mcR}{\mathcal{R}}
\newcommand*{\mcS}{\mathcal{S}}

\title{Acceleration of an iterative method for the evaluation of high-frequency multiple scattering effects}

\author{
Yassine Boubendir\thanks{New Jersey Institute of Technology, Department of Mathematical Sciences, University Heights, Newark NJ 07102, USA. boubendi@njit.edu},
Fatih Ecevit\thanks{Bo\u{g}azi\c{c}i University, Department of Mathematics, Bebek TR 34342, Istanbul, Turkey. fatih.ecevit@boun.edu.tr},
and
Fernando Reitich\thanks{CAP S.A., Gertrudis Eche\~{n}ique 220, Santiago, Chile. fernando.reitich@cap.cl}
}

\begin{document}

\date{}
\maketitle

\begin{abstract}
High frequency integral equation methodologies display the capability of reproducing
single-scattering returns in frequency-independent  computational times and employ
a Neumann series formulation to handle multiple-scattering effects. This requires
the solution of an enormously large number of single-scattering problems to attain 
a reasonable numerical accuracy in geometrically challenging configurations. Here we
propose a novel and effective Krylov subspace method suitable for the use of high frequency
integral equation techniques and significantly accelerates the convergence of Neumann
series. We additionally complement this strategy utilizing a  preconditioner based upon
Kirchhoff approximations that provides a further reduction in the overall computational
cost.
\end{abstract}

\section{Introduction}
\label{sec1}

In the last two decades significant advances have taken place in
the realm of computational scattering with notable theoretical as well as
practical contributions in the domains of finite elements
\cite{HesthavenWarburton04,DaviesEtAl09,Boffi10} and integral equations
\cite{BrunoKunyansky01,AminiProfit03,BanjaiHackbusch05,TongChew10,BrunoEtAl13}.
However, simulation strategies based upon the former are usually restricted to
low and mid frequency applications. Indeed, use of finite element
methods in exterior scattering simulations requires not only utilization
of an artificial interface to truncate the infinite computational domain
but also introduction of appropriate absorbing boundary conditions on
this interface to effectively replicate the behaviour of solution at infinity
\cite{Antoine08, EnquistMajda, Givoli04, GroteKirsch07, GroteSim11}.
This, in return, renders finite-element methods impractical in high-frequency applications
and may result in a loss of accuracy and increased computational cost.
Moreover, this difficulty is further amplified on models involving multiple scatterers,
such as the one treated in the present paper, because the distance that separates the
obstacles naturally increases the size of the truncated domain. Integral equation methods,
in contrast, are more adequate for these situations since, on the one hand, they explicitly
enforce the radiation condition by simply choosing an appropriate \emph{outgoing}
fundamental solution and, on the other hand, they are solely based on the knowledge
of solution confined only to the scatterers which, in surface scattering applications,
provides a dimensional reduction in computational domain \cite{ColtonKress92}.
Nevertheless, they deliver dense linear
systems whose sizes increase in proportion to $k^{p}$ with increasing wavenumber
$k$ where $p$ is the dimension of the computational manifold.

Broadly speaking, the success of integral equation approaches in high-frequency
simulations is directly linked with the incorporation of asymptotic characteristics of
the unknown into the solution strategy. This is essentially the path we follow in
this manuscript, since it transforms the problem into the determination of
a new unknown whose oscillations are virtually independent of frequency. While pioneering
work in this direction is due to Ned\'{e}l\'{e}c et al. \cite{AbboudEtAl94,AbboudEtAl95}
who, in two-dimensional simulations, have provided a reduction from $\mathcal{O}(k)$
to $\mathcal{O}(k^{1/3})$ in the number of degrees of freedom needed to obtain a
prescribed accuracy, the single-scattering algorithm of Bruno et al.
\cite{BrunoEtAl04} (based on a combination of  Nystr\" om method,
extensions of the method of stationary phase, and a change of variables around
the shadow boundaries) has had a significant impact as it has demonstrated the
possibility of $\mathcal{O}(1)$ solution of surface scattering problems (see
\cite{BrunoGeuzaine07} for a three-dimensional variant).
Alternative implementations of this approach built on a collocation and geometrical
theory of diffraction combo \cite{Giladi07}, a collocation and steepest descent
amalgamation \cite{HuybrechsVandewalle07}, and a p-version Galerkin interpretation
\cite{DominguezEtAl07} have later appeared. In this latter setting, Ecevit et al.
\cite{EcevitOzen16} have recently developed a rigorous method which demands, for any convex
scatterer, an increase of $\mathcal{O}(k^{\epsilon})$ (for any $\epsilon >0$) in the number
of degrees of freedom to maintain a prescribed accuracy independent of frequency.

The single-scattering algorithm \cite{BrunoEtAl04} has been successfully extended by Bruno et
al. \cite{BrunoEtAl05} to encompass the high-frequency multiple-scattering problems considered
in this paper, relating specifically to a finite collection of convex obstacles. Roughly speaking,
the approach in \cite{BrunoEtAl05} was based on: 1) Representation of the overall solution
as an infinite superposition of single scattering effects through use of a Neumann series,
2) Determination of the phase associated with each one of these effects using a spectral
geometrical optics solver, and 3) Utilization of the high-frequency single scattering algorithm
\cite{BrunoEtAl04} for the frequency independent evaluation of these effects. While every
numerical implementation in \cite{BrunoEtAl05} has displayed the spectral convergence
of Neumann series for two convex obstacles, unfortunately, a rigorous proof of this fact
was not available. Indeed, we have later shown for several convex obstacles in both
two-~\cite{EcevitReitich09} and three-dimensional~\cite{AnandEtAl10} settings that the
Neumann series can be rearranged into contributions associated with \emph{primitive periodic orbits}
and an explicit rate of convergence formula can be rigorously derived on each periodic orbit
in the high-frequency regime. While, on the one hand, these analyses depict the convergence
of Neumann series for all sufficiently large wavenumbers $k$, on the other hand, the rate
of convergence formulas display that convergence can be rather slow particularly when
(at least) one pair of nearby obstacles exists. This analysis of the rate convergence
\cite{EcevitReitich09,AnandEtAl10} was performed by using \emph{double layer} potentials. 
In this work, we show that use of \emph{combined field integral equations} lead to the same
rate of convergence. Accordingly, novel mechanisms are much
needed for the accelerated solution of multiple scattering problems that retain the
frequency independent operation count underlying the algorithm in \cite{BrunoEtAl05}.
However, this is a rather challenging task since the algorithm in \cite{BrunoEtAl05}
undeviatingly rests on reducing the problem, at each iteration, to the computation of an
unknown with a single-valued phase, and thus any strategy aimed at accelerating the
convergence of Neumann series must also preserve the phase information related with
the iterates.

In this paper, we develop a Krylov subspace method that significantly accelerates the
convergence of Neumann series, in particular in the case  where the distance between
obstacles decreases hence deteriorating the rate of convergence. This method is well
adapted to the high frequency aspect of the present problem as it retains the phase
information associated with the iterates and delivers highly accurate solutions in a small
number of iterations.
Note specifically that a direct implementation of Krylov subspace methods inhibits the use of
the algorithm in \cite{BrunoEtAl05} as this makes it impossible to track the phase information of the
corresponding iterates. As we shall see, a natural attempt to overcome this issue would be to simply
use the binomial formula, however, this disrupts the convergence of the method as displayed in the
numerical results. We defeat this additional difficulty by introducing an alternative
numerically stable decomposition of the iterates. In summary, our approach is based on three main elements:
1) Utilization of an appropriate formulation of the multiple scattering problem in the form of an operator
equation of the second kind, 2) Alternative representation of the associated Krylov subspaces
so as to guarantee that basis elements are single-phased and thus retain the frequency
independent operation count underlying the algorithm in \cite{BrunoEtAl05}, and
3) A novel decomposition of the iterates entering in a (standard) Krylov
recursion to prevent instabilities that would otherwise arise in a typical
implementation based on binomial identity. Indeed, as depicted in our numerical
implementations, the resulting methodology is immune to numerical instabilities as
it removes the additive cancellations arising from a direct use of binomial theorem.
Moreover, it provides additional savings in the number of needed iterations when
compared with the classical Pad\'{e} approximants used in \cite{BrunoEtAl05}.

We additionally complement our Krylov subspace approach utilizing a preconditioner
based upon Kirchhoff approximations to further reduce the number of iterations needed to
obtain a given accuracy. Indeed, since the knowledge of the illuminated regions at
each iteration are readily available through the geometrical optics solver we have
used to precompute the phase of multiple scattering iterations, essentially the only
additional computation needed for the application of this preconditioner is the use of
stationary phase method to deal with non-singular integrals wherein the only stationary
points are the target ones. This kind of \emph{dynamical} preconditioning is unusual
and its originality resides in the fact that the location of illuminated regions varies at
each reflection. This clearly distinguishes our preconditioning strategy from classical
approaches where the preconditioners are usually \emph{steady} by design.

While the success of this Kirchhoff preconditioner is clearly
displayed in our numerical tests, the utilization of Kirchhoff approximations for the multiple
scattering iterations naturally arises the question of convergence of the associated Neumann
series. We address this problem by showing that this series converges for each member
of a general general class of functions, and explain the exact sense in which the spectral
radius of the Kirchhoff operator is strictly less than 1. The importance of this result is twofold.
First, it verifies that the multiple scattering problem can be solved by using solely the
Kirchhoff technique, and further it rigorously answers the validity
of our preconditioning strategy.

The rest of the paper is organized as follows. In \S\ref{sec:Formulation}, we introduce the
scattering problem and provide a comparison of the equivalent differential and integral equation
formulations of multiple scattering problems. \S\ref{sec:Convergence} is reserved for a
comparison of convergence characteristics of these approaches. In \S\ref{sec:FreqIndep}, we
provide a short review of the algorithm in \cite{BrunoEtAl05} as the ideas therein lie at the core
of frequency independent evaluation of multiple scattering iterations as well as the iterates
associated with our newly proposed Krylov subspace method detailed in \S\ref{sec:Krylov}. In \S\ref{sec:Kirchhoff},
we explain how this Krylov subspace approach can be preconditioned while utilizing Kirchhoff approximations.
Finally, in \S\ref{sec:Numerical}, we present numerical implementations validating our newly proposed methodologies.


\section{Scattering problem and multiple scattering formulations}
\label{sec:Formulation}

Given an incident field $u^{\rm inc}$ satisfying the Helmholtz equation in $\mathbb{R}^n$ ($n=2,3$),
we consider the solution of sound-soft scattering problem
\begin{equation} \label{eq:sound-soft}
	\left\{ \hspace{-0.1cm}
		\begin{array}{l}
			\left( \Delta + k^{2} \right) u = 0
			\quad
			\text{in } \mathbb{R}^n \backslash \Omega \vspace{0.15cm} 
			\\
			u = -u^{\rm inc}
			\quad \text{on } \partial \Omega \vspace{0.2cm} 
			\\
			\lim_{|x| \to \infty} |x|^{(n-1)/2} \left( \partial_{\left| x \right|} - i k \right) u(x) = 0
		\end{array}
	\right.
\end{equation}
in the exterior of a smooth compact obstacle $\Omega \subset \mathbb{R}^n$. Potential theoretical
considerations entail that \cite{ColtonKress92} the \emph{scattered field} $u$ satisfying
\eqref{eq:sound-soft} admits the single-layer representation
\[
	u(x) = - \int_{\partial K} \Phi(x,y) \, \eta(y) \, ds(y)
\]
where
\begin{equation*} 
	\eta = \partial_{\nu} \left(u+u^{\rm inc}\right)
	\quad
	\text{on } \partial \Omega
\end{equation*}
is the unknown \emph{normal derivative of the total field} (called the \emph{surface current} in
electromegnatics),
$\nu$ is the exterior unit normal to ${\partial \Omega}$,
\[
	\Phi(x,y) =
	\left\{
		\begin{array}{cl}
			\dfrac{i}{4} \ H_{0}^{(1)}(k|x-y|), & n=2, \vspace{0.2cm}
			\\
			\dfrac{1}{4\pi} \dfrac{e^{ik \left| x-y\right|}}{\left| x -y \right|}, & n=3,
		\end{array}
	\right.
\]
is the fundamental solution of the Helmholtz equation, and $H_{0}^{(1)}$ is the Hankel function of
the first kind and order zero. Although $\eta$ can be recovered through a variety of integral equations
\cite{ColtonKress92}, we use the uniquely solvable \emph{combined field integral
equation} (CFIE)
\begin{equation}
\label{eq:CFIE}
	\eta (x) 
	- \int_{\partial \Omega}
	\left( \partial_{\nu(x)} +ik \right) G(x,y) \,
	\eta(y) \, ds(y)
	= f(x),
	\quad
	x \in \partial \Omega
\end{equation}
where $G = -2 \Phi$ and $f(x) = 2 \left( \partial_{\nu(x)} + ik \right) u^{\rm inc}(x)$.

In case the obstacle $\Omega$ consists of finitely many disjoint sub-scatterers $\Omega_1, \ldots, \Omega_J$,
denoting the restrictions of $\eta$ and $f$ to $\partial \Omega_j$ by $\eta_j$ and $f_j$
so that
\begin{equation} \label{eq:etaandf}
	\eta = \left( \eta_1,\ldots,\eta_J \right)^{t}
	\qquad
	\text{and}
	\qquad
	f = \left( f_1,\ldots,f_J \right)^{t},
\end{equation}
equation \eqref{eq:CFIE} gives rise to the coupled system of integral equations
\begin{equation} \label{eq:coupledIE}
	\left( \mcI - \mcS \right) \eta = f
\end{equation}
where
\begin{equation*} 
	\left( \mcS_{jj'} \eta_{j'} \right)(x)
	=\int_{\partial \Omega_{j'}} \! \! \!
	\left( \partial_{\nu(x)} + ik \right) G(x,y) \,
	\eta_{j'}(y) \, ds(y), \qquad x\in \partial \Omega_{j}.
\end{equation*}
In connection with the operator $\mcI-\mcS$, the following result will be useful in extending our
two-dimensional results in \cite{EcevitReitich09} concerning the convergence of multiple scattering
iterations to the case of CFIE.
\begin{theorem} \label{thm:diagonal}
For each $k > 0$, the diagonal operator
$\mathcal{D} = \operatorname{diag} \left( \mcI - \mcS \right) : L^2 \left( \partial \Omega \right) \to L^2 \left( \partial \Omega \right)$
is continuous with a continuous inverse. Moreover, if each $\Omega_j$ is star-like with respect to a point in its interior, then given $k_0 > 0$
there exists a constant $C_{k_0}> 0$
such that
\begin{equation} \label{eq:Dest}
	\Vert \mathcal{D}^{-1} \Vert_2 \le C_{k_0}
\end{equation}
for all $k \ge k_0$.
\end{theorem}
\begin{proof}
This is immediate since $\mathcal{D}$ is a diagonal operator and, as shown in \cite[Theorem 4.3]{Chandler-WildeMonk08},
each operator $\mcI-\mcS_{jj}$ $(j=1,\ldots,J)$ on its diagonal satisfies inequality \eqref{eq:Dest}.
\end{proof}

Multiplying equation \eqref{eq:coupledIE} with the inverse of $\mathcal{D}$ yields the equivalent operator equation of the
second kind
\begin{equation} \label{eq:OE}
	\left( \mcI - \mcT \right) \eta = g
\end{equation}
where
\begin{equation} \label{eq:T}
	\mcT_{jj'} =
	\left\{ \!\!
		\begin{array}{cl}
			0, & j=j', \\
			\left( \mcI - \mcS_{jj} \right)^{-1} \mcS_{jj'}, & j \ne j',
		\end{array}
	\right.
\end{equation}
and
\[
	g = \left( g_1,\ldots,g_J \right)^{t}
\]
with $g_j = \left( \mcI - \mcS_{jj} \right)^{-1} f_j$.
Under suitable restrictions on the geometry of scatterers, the solution of the operator equation
\eqref{eq:OE} is given by the Neumann series \cite{EcevitReitich09,AnandEtAl10}
\begin{equation} \label{eq:Neumann}
	\eta = \sum_{m=0}^{\infty} \eta^{m}
\end{equation}
where the \emph{multiple scattering iterations} 
\begin{equation} \label{eq:multscat}
	\eta^{m} = (\eta_{1}^{m}, \ldots, \eta_{J}^{m})^t
\end{equation}
are defined by
\begin{equation} \label{eq:etam}
	\eta^{m} =
	\left\{
		\begin{array}{cl}
			g, & m = 0, \\
			\mcT\eta^{m-1}, & m \ge 1.
		\end{array}
	\right.
\end{equation}

%

As was presented in \cite{Balabane04}, the multiple scattering problem described above
possesses an equivalent differential equation formulation. Naturally, the convergence analysis
carried out in \cite{Balabane04} is directly linked with that of the Neumann series \eqref{eq:Neumann}
and here we present the exact connection. Indeed, the fields $u_j$ given by the single-layer potentials
\[
	u_j(x) = - \int_{\partial \Omega_j} \Phi(x,y) \, \eta_j(y) \, ds(y)
\]
in connection with the components of $\eta$ in \eqref{eq:etaandf} correspond precisely to the
unique solutions of the exterior sound-soft scattering problems
\[
	\left\{ \hspace{-0.1cm}
		\begin{array}{l}
			\left( \Delta + k^{2} \right) u_j= 0
			\quad \text{in } \mathbb{R}^n \backslash \Omega_j, \vspace{0.15cm} 
			\\
			u_{j} = -u^{\rm inc} - \sum_{j' \ne j} u_{j'} 
			\quad \text{on } \partial \Omega_j, \vspace{0.2cm} 
			\\
			\lim_{|x| \to \infty} |x|^{(n-1)/2} \left( \partial_{\left| x \right|} - i k \right) u_j(x) = 0,
		\end{array}
	\right.
\]
and they provide the decomposition of the scattered field $u$ as
\begin{equation} \label{eq:decompose_u}
	u = \sum_{j=1}^{J} u_j.
\end{equation}
On the other hand, the iterated fields $u_{j}^{m}$ given by the single-layer potentials
\begin{equation} \label{eq:umj}
	u_j^{m}(x) = - \int_{\partial \Omega_j} \Phi(x,y) \, \eta_j^{m}(y) \, ds(y)
\end{equation}
in relation with the components of $\eta^{m}$ in \eqref{eq:multscat} are precisely the unique
solutions of the exterior sound-soft scattering problems
\[
	\left\{ \hspace{-0.1cm}
		\begin{array}{l}
			\left( \Delta + k^{2} \right) u_j^{m} = 0
			\quad \text{in } \mathbb{R}^n \backslash \Omega_j, \vspace{0.15cm} 
			\\
			u_{j}^{m} = -h_{j}^{m} 
			\quad \text{on } \partial \Omega_j,  \vspace{0.2cm} 
			\\
			\lim_{|x| \to \infty} |x|^{(n-1)/2} \left( \partial_{\left| x \right|} - i k \right) u_{j}^{m}(x) = 0
		\end{array}
	\right.
\]
with
\[
	h_{j}^{m} =
	\left\{\hspace{-0.1cm}
		\begin{array}{ll}
			u^{\rm inc}, & m=0, \vspace{0.2cm}
			\\
			\sum_{j' \ne j} u_{j'}^{m-1}, & m \ge 1,
		\end{array}
	\right.
\]
and thus, in case the Neumann series \eqref{eq:Neumann} converges, each solution $u_j$ can
be expressed as the superposition
\begin{equation} \label{eq:series}
	u_{j} = \sum_{m=0}^{\infty} u_{j}^{m}.
\end{equation}


\section{Convergence of multiple scattering iterations}
\label{sec:Convergence}

Preliminary work on the justification of identity \eqref{eq:series} in a three-dimensional setting has
appeared in \cite{Balabane04}. Indeed, while \cite[Theorem 1]{Balabane04} establishes uniqueness
of decomposition \eqref{eq:decompose_u}, \cite[Theorem 3]{Balabane04} justifies the convergence of
the series in \eqref{eq:series} under suitable restrictions on the geometry of the obstacles $\Omega_j$
as stated in the next theorem.
\begin{theorem}[cf. \cite{Balabane04}] \label{thm:Balabane}
Assume that $u^{\rm inc} \in H^{1}(\partial \Omega)$ and, for $j=1,\ldots, J$, the obstacle $\Omega_j$
is non-trapping in the sense that 
\begin{equation*}
	\beta_j = \dfrac{1}{\operatorname{diam} \left( \Omega_j \right)} \, \sup_{y \in \Omega_j} \inf_{x \in \partial \Omega_j} \nu(x) \cdot \left( x-y \right) > 0.
\end{equation*}
Let
\[
	\delta = \max_{1\le j \le J} \operatorname{diam} \left( \Omega_j \right),
	\quad
	d = \min_{1 \le j, j' \le J} \operatorname{dist} \left( \Omega_j, \Omega_{j'} \right),
	\quad
	\left| \partial \Omega \right| = \text{surface area of } \partial \Omega.  
\]
Then there exists a constant $\beta > 0$ that depends on $\beta_1,\ldots,\beta_J$ such that if
\[
	\dfrac{\delta \, d^2}{\left| \partial \Omega \right|} \, \beta
	> k^3 \left( 1+ (\delta k)^2 \right) \sqrt{1+2 \, ( \delta k)^2},
\]
then, for $j=1,\ldots,J$, identity \eqref{eq:series} holds in the sense of convergence in
$H^{1}_{\text{loc}} \left( \mathbb{R}^n \backslash \Omega_j \right)$.
\end{theorem}
As is clear, Theorem \ref{thm:Balabane} establishes convergence of the series \eqref{eq:series} for non-trapping
obstacles only if the wavenumber $k$ is sufficiently small. Work on rigorous justification of
the convergence of Neumann series \eqref{eq:Neumann} (and thus of identity \eqref{eq:series})
in high-frequency applications, on the other hand, reduces to our work \cite{EcevitReitich09} and
\cite{AnandEtAl10} that relates to a finite collection of smooth strictly convex (and thus non-trapping
in the sense of Theorem \ref{thm:Balabane}) obstacles in two- and three-dimensions, respectively. Indeed, as we have
shown in \cite{EcevitReitich09,AnandEtAl10}, the Neumann series \eqref{eq:Neumann} can be
rearranged into a sum over \emph{primitive periodic orbits} and a precise (asymptotically geometric)
rate of convergence $\mathcal{R}_{p}$ (where $p$ is the period of the orbit), that depends only
on the relative geometry of the obstacles $\Omega_j$, can be derived on each periodic orbit in the
asymptotic limit as $k \to \infty$.

To review these results, for the sake of simplicity of exposition, we assume that the scatterer $\Omega$
consists only of two smooth strictly convex obstacles $\Omega_1$ and $\Omega_2$ in which case there
are only two (primitive) periodic orbits (initiating from each $\Omega_j$ and traversing the obstacles in a
2-periodic manner) and relation \eqref{eq:etam} is equivalent to
\begin{equation} \label{eq:describe}
	\mathcal{D} \eta^{m}
	= f^{m},
	\qquad
	\left( m \ge 0 \right),
\end{equation}
where
\[
	f^{0} = f = 2 \left( \partial_{\nu} + ik \right) u^{\rm inc},
	\quad
	\text{on } \partial \Omega,
\]
and, for $m \ge 1$,
\begin{equation} \label{eq:fm}
	f^{m}=
	\begin{bmatrix}
		0 & \mcS_{12} \\
		\mcS_{21}  & 0
	\end{bmatrix} 
	\eta^{m-1}.
\end{equation}
In connection with identity \eqref{eq:describe}, Theorem \ref{thm:diagonal} implies in
two-dimensional configurations that, given any $k_0 >0$, there exists $C_{k_0} >0$
such that for any $k \ge k_0$
\begin{equation} \label{eq:etageomrate}
	\Vert \eta^{m+2} - \mathcal{R} \eta^{m} \Vert_{L^2 \left( \partial \Omega \right)}
	\le C_{k_0} \, \Vert f^{m+2} - \mathcal{R} f^{m} \Vert_{L^2 \left( \partial \Omega \right)}
\end{equation}
holds for any constant $\mathcal{R} \in \mathbb{C}$, and thus the aforementioned \emph{geometric rate of convergence} of the
Neumann series \eqref{eq:Neumann} is directly linked with that of the right-hand sides $f^{m}$.
Indeed, assuming that the incidence is a plane-wave $u^{\rm inc}(x) = e^{ik \, \alpha \cdot x}$
with direction $\alpha$ ($\left| \alpha \right| = 1$) with respect to which the obstacles $\Omega_1$
and $\Omega_2$ satisfy the \emph{no-occlusion condition} (which amounts to requiring that
there is at least one ray with direction $\alpha$ that passes between $\Omega_1$ and $\Omega_2$
without touching them), denoting by $a_j \in \partial \Omega_j$ the uniquely determined points
minimizing the distance between $\Omega_1$ and $\Omega_2$, and setting $d = \left| a_1 - a_2 \right|$,
we have the following relation among the leading terms $f_A^{m}$ in the asymptotic expansions
of $f^{m}$ which extends our analyses in \cite{EcevitReitich09,AnandEtAl10} to the case of CFIE.

\begin{theorem} \label{thm:2per}
There exist constants $C = C \left( \Omega, \alpha \right) > 0$,
$\delta = \delta \left( \Omega,\alpha \right) \in \left( 0, 1 \right)$, and
$\mathcal{R}_{2} = \mathcal{R}_2 \left( \Omega, k \right) \in \mathbb{C}$
with the property that, for all $m \ge 1$,
\begin{equation} \label{eq:geomrate}
	\Vert f_A^{m+2}-\mathcal{R}_{2} f_A^{m} \Vert_{L^2 \left( \partial \Omega \right)}
	\le C \, k \, \delta^{m}.
\end{equation}
The constant $\mathcal{R}_2$ is given in two-dimensional configurations by
\begin{equation*}
	\mathcal{R}_2 = e^{2ikd} 
	\left(
		\sqrt{\left( 1 + d \kappa_{1} \right) \left( 1 + d \kappa_{2} \right)} \times
		\left[ 1 + \sqrt{1 - \left[ \left( 1 + d \kappa_{1})(1 + d \kappa_{2} \right) \right]^{-1}}\ \right]
	\right)^{-1}
\end{equation*}
where $\kappa_j$ is the curvature at the point $a_j$; and in three-dimensional
configurations
\begin{equation*}
	\mathcal{R}_2 = e^{2ikd} 
	\left(
		\sqrt{ \det \left[ \left( I + d \kappa_{1} \right) \left( I + d \kappa_{2} \right) \right]} \times
		\det \left[ I + \sqrt{I - \left[ T \left( I + d \kappa_{1} \right) T^{-1} \left( I + d \kappa_{2} \right) \right]^{-1}}\ \right]
	\right)^{-1}
\end{equation*}
where $I$ is the identity matrix,
\[
	\kappa_{j} =
	\begin{bmatrix}
		\kappa_{1} (a_j) & 0
		\\
		0 & \kappa_{2} (a_j)
	\end{bmatrix}
\]
is the matrix of principal curvatures at the point $a_j$, and $T$ is the rotation matrix
determined by the relative orientation of the surfaces $\partial \Omega_j$ at the points $a_j$.
\end{theorem}
\begin{proof}
Assume first that the dimension is $n=2$. Writing $f^{m} = \left[ f^{m}_1 \, f^{m}_2 \right]^t$
and $f^{m}_A = \left[ f^{m}_{A,1} \, f^{m}_{A,2} \right]^t$, it suffices to show that, for $\ell = 1,2$,
\begin{equation} \label{eq:asest}
	\Vert f^{(m+2)}_{A,\ell}-\mathcal{R}_{2} f^{m}_{A,\ell} \Vert_{L^2 \left( \partial \Omega_{\ell} \right)}
	\le C_{\ell} \, k \, \delta^{m}
\end{equation}
for some constant $C_{\ell} = C_{\ell} \left( \Omega, \alpha \right)$. On the other hand,
\cite[Theorems 3.4 and 4.1]{EcevitReitich09} display that (a more general version of) this estimate
holds on any compact subset of the \emph{illuminated regions} (see the next section
for a precise definition of these regions) when $f^{m+2}_{A,\ell}$ and $f^{m}_{A,\ell}$ are replaced
by the leading terms $\eta^{m+1}_{A,\ell}$ and $\eta^{m-1}_{A,\ell}$ in the asymptotic expansions of
$\eta^{m+1}_{\ell}$ and $\eta^{m-1}_{\ell}$. Finally applying the stationary phase lemma \cite{Fedoryuk71}
to each component of identity \eqref{eq:fm}, the same techniques used to prove
\cite[Theorem 4.1]{EcevitReitich09} delivers estimate \eqref{eq:asest}.
In case $n=3$, the argument is the same and is based upon \cite[Theorems 3.3 and 4.3]{AnandEtAl10}.
\end{proof}

Although Theorem \ref{thm:2per} is valid under the no occlusion condition, extensive numerical tests
in \cite{EcevitReitich09,AnandEtAl10} display that the conclusion of Theorem \ref{thm:2per} is valid
not only when this condition is violated but also when the convexity assumption is conveniently
relaxed.

\begin{remark}
In light of estimates \eqref{eq:etageomrate}-\eqref{eq:geomrate}, for $M \gg \log k$,
we have
\[
	\eta
	= \sum_{m=0}^{\infty} \eta^{m}
	\sim \sum_{m=0}^{M} \eta^{m}
	+ \left( \eta^{M+1} + \eta^{M+2} \right)
	\sum_{m=0}^{\infty} \mathcal{R}_2^m
\]
which signifies that the Neumann series converges with the geometric rate $\mathcal{R}_2$. Note however that,
as the distance between the obstacles $\Omega_1$ and $\Omega_2$ decreases to
zero, $\left| \mathcal{R}_2 \right|$ increases to $1$, and thus convergence of the Neumann series significantly
deteriorates.
\end{remark}

The same remark is valid when the configuration consists of more than two subscaterers and involves
at least one pair of nearby obstacles. Indeed, as we have shown in \cite{EcevitReitich09,AnandEtAl10}, this is
also completely transparent from a theoretical point of view since, in this case, the Neumann series can be
completely dismantled into single-scattering effects and rearranged into a sum over \emph{primitive periodic
orbits} including, in particular, $2$-periodic orbits.

The next section is devoted to the description of how we adopt the high-frequency integral equation
method in \cite{BrunoEtAl05} to the evaluation of iterates arising in our Krylov subspace approach and
also in its preconditioning through Kirchhoff approximations. As explained in the introduction, the strength
of the work in \cite{BrunoEtAl05} is due to retaining information on the phases of multiple scattering iterations,
and therefore our Krylov subspace and Kirchhoff preconditioning strategies are also designed to posses
the same property.


\section{High-frequency integral equations for multiple scattering configurations}
\label{sec:FreqIndep}

For simplicity of exposition we continue to assume that the obstacle $\Omega$ consists only of two disjoint
sub-scatterers $\Omega_1$ and $\Omega_2$. In what follows, for $j, j' \in \{1,2\}$, we will always assume that
$j \ne j'$. In this case, relation \eqref{eq:describe} can be written,
for $j = 1,2$, in components as
\begin{equation} \label{eq:etanorefl}
	\left( \mcI - \mcS_{jj} \right) \eta_j^{0} = f^{0}_j
	\quad
	\text{on } \partial \Omega_j,
\end{equation}
and
\begin{equation} \label{eq:etarefl}
	\left( \mcI - \mcS_{jj} \right) \eta_j^{m} = \mcS_{jj'} \, \eta_{j'}^{m-1} 
	\quad
	\text{on } \partial \Omega_j,
\end{equation}
for $m \ge 1$.
As identity \eqref{eq:etanorefl} displays, $\eta^{0}_j$ is exactly the surface current generated by the incidence
$u^{\rm inc}$ on $\partial \Omega_j$ ignoring interactions between $\Omega_1$ and $\Omega_2$. Similarly, for
$m \ge 1$, equation \eqref{eq:etarefl} depicts that $\eta^{m}_j$ is precisely the surface current generated by the field $u^{m-1}_{j'}$ (note that
$\mcS_{jj'} \, \eta^{m-1}_{j'} = 2 \left( \partial_{\nu} + ik \right) u^{m-1}_{j'}$) acting as an incidence on
$\partial \Omega_j$ ignoring, again, interactions between $\Omega_1$ and $\Omega_2$.
Therefore identities \eqref{eq:etanorefl} and \eqref{eq:etarefl} entail that the Neumann series \eqref{eq:Neumann} completely dismantles the single
scattering contributions and allows for a representation of the surface current $\eta$ as a superposition
of these effects. More importantly, in geometrically relevant configurations, these observations allow us
to predetermine the phase $\varphi^{m}_j$ of $\eta^{m}_j$ and express it as the product of a highly
oscillating complex exponential modulated by a slowly varying amplitude in the form
\begin{equation} \label{eq:factor}
	\eta^{m}_j
	= e^{ik \, \varphi^{m}_j} \ \eta^{m, \, {\rm slow}}_j 
\end{equation}
and this, in turn, grants the frequency-independent solution of equations \eqref{eq:etanorefl}-\eqref{eq:etarefl}
as described in \cite{BrunoEtAl05}. To review the algorithm in \cite{BrunoEtAl05} and set the stage in the
rest of the paper, we first describe the phase functions $\varphi^{m}_j$ in combination with the
various regions they determine on the boundary of the scatterers, and we present one of the main
results in \cite{EcevitReitich09,AnandEtAl10} that displays the asymptotic characteristics of the amplitudes
$\eta^{m, \, {\rm slow}}_j$.

Indeed, in case the obstacles $\Omega_1$ and $\Omega_2$ are convex and satisfy the no occlusion
condition with respect to the direction of incidence $\alpha$, the phase $\varphi^{m}_j$ in \eqref{eq:factor}
is given by
\begin{equation} \label{eq:phases}
	\varphi^{m}_j =
	\left\{
		\begin{array}{ll}
			\phi^{m}_j, & m \text{ is even}, \vspace{0.2cm}
			\\
			\phi^{m}_{j'}, & m \text{ is odd}.
		\end{array}
	\right.
\end{equation}
Here, for any of the two \emph{obstacle paths} $\{ \Gamma^m_1 \}_{m \ge 0}$
and $\{ \Gamma^m_2 \}_{m \ge 0}$ defined by
\[
	\left( \Gamma_1^{2m} , \Gamma_1^{2m+1} \right) = \left( \partial \Omega_{1} , \partial \Omega_{2} \right)
	\qquad
	\text{and}
	\qquad
	\left( \Gamma_2^{2m} , \Gamma_2^{2m+1} \right) = \left( \partial \Omega_{2} , \partial \Omega_{1} \right)
\]
for all $m \ge 0$, the geometrical phase $\phi^{m}_{\ell}$ at any point $x \in \Gamma^m_{\ell}$ ($\ell=1,2$) is uniquely
defined as \cite{EcevitReitich09,AnandEtAl10}
\[
	\phi^{m}_{\ell}(x)= 
	\left\{ \! \!
		\begin{array}{ll}
			\alpha \cdot x,
			&
			m = 0, \\
			\alpha \cdot \mathcal{X}^{m}_{0}(x)
			+ \sum\limits_{r=0}^{m-1} |\mathcal{X}^{m}_{r+1}(x)-\mathcal{X}^{m}_{r}(x)|,
			& m \ge 1 \, ,
		\end{array}
	\right.
\]
where the points $(\mathcal{X}^{m}_0(x), \ldots, \mathcal{X}^{m}_{m}(x)) \in \Gamma^0_{\ell} \times \cdots \times \Gamma^{m}_{\ell}$
are specified by
\begin{align*}
	& \text{(a)} \
	\mathcal{X}^{m}_{m}(x) = x,
	\vspace{0.2cm}
	\\
	& \text{(b)} \
	\alpha \cdot \nu(\mathcal{X}^{m}_{0}(x)) < 0,
	\vspace{0.2cm}
	\\
	& \text{(c)} \
	(\mathcal{X}^{m}_{r+1}(x)-\mathcal{X}^{m}_{r}(x))
	\cdot \nu(\mathcal{X}^{m}_{r}(x)) > 0,
	\vspace{0.2cm}
	\\
	& \text{(d)} \
	\dfrac{\mathcal{X}^{m}_{1}(x)-\mathcal{X}^{m}_{0}(x)}
	{|\mathcal{X}^{m}_{1}(x)-\mathcal{X}^{m}_{0}(x)|}
	= \alpha - 2 \alpha \cdot \nu(\mathcal{X}^{m}_{0}(x)) \, \nu(\mathcal{X}^{m}_{0}(x)),
	\vspace{0.9cm}
	\\
	& \text{(e)} \
	\dfrac{\mathcal{X}^{m}_{r+1}(x)-\mathcal{X}^{m}_{r}(x)}
	{|\mathcal{X}^{m}_{r+1}(x)-\mathcal{X}^{m}_{r}(x)|}
	= \dfrac{\mathcal{X}^{m}_{r}(x)-\mathcal{X}^{m}_{r-1}(x)}
	{|\mathcal{X}^{m}_{r}(x)-\mathcal{X}^{m}_{r-1}(x)|}
	- 2 \dfrac{\mathcal{X}^{m}_{r}(x)-\mathcal{X}^{m}_{r-1}(x)}
	{|\mathcal{X}^{m}_{r}(x)-\mathcal{X}^{m}_{r-1}(x)|}
	\cdot \nu(\mathcal{X}^{m}_{r}(x)) \, \nu(\mathcal{X}^{m}_{r}(x))
\end{align*}
for $0 < r <m$. These conditions simply mean the phase $\phi^{m}_{\ell}(x)$ is determined
by the ray with initial direction $\alpha$ sequentially hitting at and bouncing off the points
$\mathcal{X}^{m}_{r}(x)$ ($r=0,\ldots,m-1$) according to the law of reflection to finally
arrive at $x \in \Gamma^m_{\ell}$. Moreover, these rays divide $\Gamma^m_{\ell}$ into two
open connected subsets, namely the \emph{illuminated regions}
\[
	\Gamma^{m}_{\ell}(IL)
	= \left\{
		\begin{array}{ll}
			\left\{
				x \in \Gamma_{\ell}^{0} :
				\alpha \cdot \nu(x) < 0
			\right\},
			& \quad m = 0,
			\\
			\\
			\left\{
				x \in \Gamma_{\ell}^{m} :
				(\mathcal{X}^{m}_{m}(x)-\mathcal{X}^{m}_{m-1}(x))
				\cdot \nu(x) < 0
			\right\},
			& \quad m \ge 1,
		\end{array}
	\right.
\]
and the \emph{shadow regions}
\[
	\Gamma^{m}_{\ell}(SR)
	= \left\{
		\begin{array}{ll}
			\left\{
				x \in \Gamma^{0}_{\ell} :
				\alpha \cdot \nu(x) > 0
			\right\},
			& \quad m = 0,
			\\
			\\
			\left\{
				x \in \Gamma^{m}_{\ell} :
				(\mathcal{X}^{m}_{m}(x)-\mathcal{X}^{m}_{m-1}(x))
				\cdot \nu(x) > 0
			\right\},
			& \quad m \ge 1,
		\end{array}
	\right.
\]
and their closures intersect at the \emph{shadow boundaries}
\[
	\Gamma^{m}_{\ell}(SB)
	= \left\{
		\begin{array}{ll}
			\left\{
				x \in \Gamma^{0}_{\ell} :
				\alpha \cdot \nu(x) = 0
			\right\},
			& \quad m = 0,
			\\
			\\
			\left\{
				x \in \Gamma^{m}_{\ell} :
				(\mathcal{X}^{m}_{m}(x)-\mathcal{X}^{m}_{m-1}(x))
				\cdot \nu(x) = 0
			\right\},
			& \quad m \ge 1,
		\end{array}
	\right.
\]
each of which consists of two points in two-dimensional configurations or a smooth
closed curve in three-dimensions. In connection with the phase functions
\eqref{eq:phases}, illuminated regions $\partial \Omega^{m}_{j}(IL)$,
shadow regions $\partial \Omega^{m}_{j}(SR)$, and the shadow
boundaries $\partial \Omega^{m}_{j}(SB)$ are then given by
\[
	\partial \Omega^{m}_{j}(\, \cdot \,) = 
	\left\{
		\begin{array}{ll}
			\Gamma^{m}_{j}(\, \cdot \,), & m \text{ is even}, \vspace{0.2cm}
			\\
			\Gamma^{m}_{j'}(\, \cdot \,), & m \text{ is odd}.
		\end{array}
	\right.
\]
Generally speaking this means that the rays emanating from $\partial \Omega^{m}_{j}$
return to $\partial \Omega^{m}_{j}$ after an even number of reflections, and those initiating
from $\partial \Omega^{m}_{j'}$ arrive $\partial \Omega^{m}_{j}$ after an odd number of
reflections. Finally let us note that the phase functions $\phi^{m}_{j}$ are smooth and
periodic as they are confined to the boundary of the associated scatterers. The computation
of these phases are performed using a spectrally accurate geometrical optics solver. This also allows for
a simple and accurate determination of the shadow boundary points and thus the illuminated
and shadow regions.

With these definitions we can now state one of the
main results in \cite{EcevitReitich09,AnandEtAl10} that completely clarifies the asymptotic  
behavior of amplitudes $\eta^{m, \, \rm{slow}}_{j}$ in \eqref{eq:factor}.

\begin{theorem}[\cite{EcevitReitich09,AnandEtAl10}]
\label{thm:hormander}
\begin{itemize}
\item[{\bf (i)}] On the illuminated region $\partial \Omega^{m}_{j}(IL)$,
$\eta^{m, \, \rm{slow}}_{j}(x, k)$ belongs to the H\"{o}rmander class
$S^{1}_{1,0}(\partial \Omega^{m}_{j}(IL) \times (0,\infty))$ (cf. \cite{Hormander66,Hormander71c})
and admits the asymptotic expansion
\begin{equation}
\label{eq:MT85anew}
	\eta^{m, \, \rm{slow}}_{j}(x, k)
	\sim
	\sum_{p \ge 0} k^{1-p} a^{m}_{j,p}(x)
\end{equation}
where $a^{m}_{j,p}$ are complex-valued $C^{\infty}$
functions. Consequently, for any $P \in \mathbb{N} \cup \{ 0 \}$, the difference
\begin{equation}
\label{eq:rmN}
	r^{m}_{j,P}(x,k) = \eta^{m, \, \rm{slow}}_{j}(x, k) - \sum_{p = 0}^{P} k^{1-p} a^{m}_{j,p}(x)
\end{equation}
belongs to $S^{-P}_{1,0}(\partial \Omega^{m}_{j}(IL) \times (0,\infty))$
and thus satisfies the estimates
\begin{equation}
\label{eq:ILestnew}
	\left| D_{x}^{\beta} D_{k}^{n} r^{m}_{j,N}(x,k) \right| \le C_{m,\beta,n,S} (1+k)^{-P-n}
\end{equation}
on any compact subset $S$ of $\partial \Omega^{m}_{j}(IL)$ for any multi-index $\beta$
and $n \in \mathbb{N} \cup \{ 0 \}$.

\item[{\bf (ii)}] Over the entire boundary $\partial \Omega_{j}$, $\eta^{m, \, \rm{slow}}_{j}(x, k)$
belongs to the H\"{o}rmander class $S^{1}_{2/3,1/3}(\partial \Omega_{j} \times (0,\infty))$
and admits the asymptotic expansion
\begin{equation}
\label{eq:MT85new}
	\eta^{m, \, \rm{slow}}_{j}(x, k)
	\sim
	\sum_{p,q \ge 0} k^{2/3-2p/3-q} \, b^{m}_{j,p,q}(x) \Psi^{(p)}(k^{1/3}Z^{m}_{j}(x))
\end{equation}
where $b^{m}_{j,p,q}(x)$ are complex-valued $C^{\infty}$ functions, $Z^{m}_{j}(x)$
is a real-valued $C^{\infty}$ function that is positive on $\partial \Omega^{m}_{j}(IL)$,
negative on $\partial \Omega^{m}_{j}(SR)$, and vanishes precisely to first order on
$\partial \Omega^{m}_{j}(SB)$, and the function $\Psi$ admits the asymptotic expansion
\begin{equation} \label{eq:PsiIL}
	\Psi(\tau) \sim \sum_{\ell \ge 0} c_{\ell} \tau^{1-3\ell}
	\qquad
	\text{as } \tau \to \infty,
\end{equation}
and it is rapidly decreasing in the sense of Schwartz as $\tau \to -\infty$.
Note specifically then, for any $P,Q \in \mathbb{N} \cup \{ 0 \}$, the difference
\begin{equation*}
	R^{m}_{P,Q}(x,k)
	= \eta^{m, \, \rm{slow}}_{j}(x, k)
	- \sum_{p,q = 0}^{P,Q} k^{2/3-2p/3-q} \, b^{m}_{j,p,q}(x) \Psi^{(p)}(k^{1/3}Z^{m}_{j}(x))
\end{equation*}
belongs to $S^{-\mu}_{2/3,1/3}(\partial \Omega_{j} \times (0,\infty))$,
$\mu = \min \left\{ 2P/3,Q \right\}$, and thus satisfies the estimates
\begin{equation}
\label{eq:estentirenew}
	\left| D_{x}^{\beta} D_{k}^{n} R^{m}_{P,Q}(x,k) \right| \le C_{m,\beta,n} (1+k)^{-\mu-2n/3+|\beta|/3}
\end{equation}
for any multi-index $\beta$ and $n \in \mathbb{N} \cup \{ 0 \}$.
\end{itemize}
\end{theorem}

The first main ingredient underlying the algorithm in \cite{BrunoEtAl05} was
the observation that, while $\eta^{m, \, \rm{slow}}_{j}$ admits a classical asymptotic
expansion in the illuminated region $\partial \Omega^{m}_{j}(IL)$ as displayed by
equation \eqref{eq:MT85anew}, it possesses boundary layers of order
$\mathcal{O}(k^{-1/3})$ around the shadow boundaries $\partial \Omega^{m}_{j}(SB)$
and rapidly decays in the shadow region $\partial \Omega^{m}_{j}(SR)$
as implied by the expansion \eqref{eq:MT85new} and the mentioned change in the asymptotic
expansions of the function $\Psi$.
Therefore, as depicted in \cite{BrunoEtAl05}, utilizing a cubic root change of variables
in $k$ around the shadow boundaries, the unknown $\eta^{m, \, \rm{slow}}_{j}$ can be
expresses in a number of degrees of freedom independent of frequency, and this
transforms the problem into the evaluation of highly oscillatory integrals.

Indeed, a second main element of the algorithm in \cite{BrunoEtAl05} is based on
the realization that the identity
\begin{equation} \label{eq:H10der}
	\dfrac{d}{dz} H^{(1)}_{0}(z) = - H^{(1)}_{1}(z)
\end{equation}
combined with the asymptotic expansions of Hankel functions \cite{AbramowitzStegun14}
entails
\[
	\left( \partial_{\nu(x)} +ik \right) G(x,y)
	\sim
	e^{ik \, |x-y|}
	\left( 
		e^{-i\pi/4}
		\left( \dfrac{k}{2\pi \, |x-y|} \right)^{1/2}
		\left( 1+ \dfrac{x-y}{|x-y|} \cdot \nu(x) \right)
	\right),
\]
and thus, in light of factorization \eqref{eq:factor}, equations \eqref{eq:etanorefl}-\eqref{eq:etarefl}
take on the form
\begin{equation} \label{eq:CFIEslowzero}
	e^{ik \, \varphi^{0}_{j}(x)} \ \eta^{0, \, {\rm slow}}_{j}(x) 
	- \int_{\partial \Omega_j}
	e^{ik \, (\varphi^{0}_{j}(y) + |x-y|)} \,
	F(x,y) \
	\eta^{0, \, {\rm slow}}_{j}(y)  \, ds(y)
	= f^0_j(x),
	\quad
	x \in \partial \Omega^{0}_{j},
\end{equation}
and, for $m \ge 1$,
\begin{multline} \label{eq:CFIEslowm}
	e^{ik \, \varphi^{m}_{j}(x)} \ \eta^{m, \, {\rm slow}}_{j}(x) 
	- \int_{\partial \Omega_j}
	e^{ik \, (\varphi^{m}_{j}(y) + |x-y|)} \,
	F(x,y) \,
	\eta^{m, \, {\rm slow}}_{j}(y)  \, ds(y)
	\\
	= \int_{\partial \Omega_{j'}}
	e^{ik \, (\varphi^{m-1}_{j}(y) + |x-y|)} \,
	F(x,y) \,
	\eta^{m-1, \, {\rm slow}}_{j'}(y)  \, ds(y),
	\quad
	x \in \partial \Omega^{m}_{j},
\end{multline}
where
\[
	F(x,y) = e^{-ik \, |x-y|} \left( \partial_{\nu(x)} + ik \right) G(x,y).
\]
As depicted in \cite{BrunoEtAl05}, frequency independent evaluations of integrals in
\eqref{eq:CFIEslowzero}-\eqref{eq:CFIEslowm} can then be accomplished to any desired accuracy
utilizing a \emph{localized integration} (around stationary points of the combined phase $\varphi^{m}_{j}(y) + |x-y|$ 
and/or the singularities of the integrand) procedure based upon suitable extensions of the method of
stationary phase.

The third main element of the algorithm in \cite{BrunoEtAl05} is the use of Nysr\"{o}m and trapezoidal
discretizations and Fourier interpolations to render the method high order, and the scheme is finally completed
with a matrix-free Krylov subspace linear algebra solver to obtain accelerated solutions.    

While the above discussion provides a brief summary of the algorithm in \cite{BrunoEtAl05}, it clearly
signifies the importance of retaining the phase information in connection with the multiple scattering
iterations since this allows for a simple utilization of the aforementioned localized integration scheme.
Accordingly, any strategy aiming at accelerating the convergence of Neumann series must also preserve
the phase information. As we explain, both the novel Krylov subspace method we develop in the next
section and its preconditioning discussed in section \ref{sec:Kirchhoff} posses this property.

\section{Novel Krylov subspace method for accelerating the convergence of Neumann series}
\label{sec:Krylov}

As with the solution of matrix equations, Krylov subspace methods provide a convenient mechanism for the approximate
solution of operator equations
\[
	\mcA \eta = g
\] 
in Hilbert spaces (see e.g. \cite{Saad03} and the references therein). These methods are \emph{orthogonal projection
methods} wherein, given an initial approximation $\mu^{(0)}$ to $\eta$, one seeks an approximate solution $\mu^{(m)}$ from the
affine space $\mu^{(0)} + K_m$ related with the \emph{Krylov subspace}
\[
	K_m
	= \operatorname{span} \, \{ r^{(0)}, \mcA r^{(0)}, \mcA^2r^{(0)}, \ldots, \mcA^{m-1} r^{(0)} \}
\]
of the operator $\mathcal{A}$ associated with the \emph{residual} $r^{(0)} = g - \mcA \mu^{(0)}$ imposing the
\emph{Petrov-Galerkin condition} 
\[
	g - \mcA \mu^{(m)} \perp K_m.
\]

In connection with the operator equation \eqref{eq:OE}, taking $\mu^{(0)} = 0$,
the approximate solution $\mu^{(m)}$ belongs to the Krylov subspace
\[
	K_m
	= \operatorname{span} \, \{ g, (\mcI-\mcT)g, (\mcI-\mcT)^2g, \ldots, (\mcI-\mcT)^{m-1} g \}
\]
for which, in light of identity \eqref{eq:etam}, the functions $\left( \mcI - \mcT \right)^n g$ can be expressed as
linear combinations of the multiple scattering iterations $\eta^{\ell}$ through use of the binomial theorem as
\begin{equation} \label{eq:binom}
	\left( \mcI - \mcT \right)^n g
	= \sum_{\ell=0}^{n} \binom{n}{\ell} \left( -1 \right)^{\ell} \mcT^{\ell} \, g
	= \sum_{\ell=0}^{n} \binom{n}{\ell} \left( -1 \right)^{\ell} \eta^{\ell}.
\end{equation}
This relation clearly entails
\[
	K_m
	= \operatorname{span} \, \{ \eta^{0},\ldots,\eta^{m-1} \}
\]
and thus, any information about the Krylov subspace $K_m$ can be obtained in frequency
independent computational times using the algorithm briefly described in \S4.

A particular Krylov subspace method we favor for the solution of multiple scattering problem \eqref{eq:OE}
is the classical ORTHODIR \cite{Saad03} iteration which, for the initial guess $\mu^{(0)} = 0$, takes on the form
\[
	\begin{array}{l}
		\text{1. Set } r^{(0)} = p^{(0)}= g, \vspace{0.1cm} \\
		\text{2. For } j = 0,1\dots \text{  DO}  \vspace{0.1cm} \\
		\qquad \text{2.1 } \ \alpha_j = \langle r^{(j)}, \mathcal{A} p^{(j)} \rangle / \langle \mathcal{A} p^{(j)}, \mathcal{A}p^{(j)} \rangle, \vspace{0.1cm} \\
		\qquad \text{2.2 } \ \mu^{(j+1)} = \mu^{(j)} + \alpha_j \, p^{(j)}, \vspace{0.1cm} \\
		\qquad \text{2.3 } \ r^{(j+1)} = r^{(j)} - \alpha_j \, \mathcal{A} p^{(j)}, \vspace{0.1cm} \\
		\qquad \text{2.4 } \ \text{For } i = 0,\dots,j, \
		\beta_{ij} = - \langle \mathcal{A}^2p^{(j)}, \mathcal{A}p^{(i)} \rangle / \langle \mathcal{A}p^{(i)}, \mathcal{A}p^{(i)} \rangle, \vspace{0.1cm} \\
		\qquad \text{2.5 } \ p^{(j+1)} = \mathcal{A}p^{(j)} + \sum_{i=0}^{j} \beta_{ij} \, p^{(i)}.
	\end{array}
\]
This iteration entails, through a straightforward induction argument, the following recurrence relation for 
$\mathcal{A} = \mcI - \mcT$ where $\mcT$ is the iteration operator specified by equation \eqref{eq:T}.
\begin{theorem} For $\mathcal{A} = \mcI - \mcT$, the iterates $p^{(j)}$ generated by the ORTHODIR algorithm satisfy the recurrence
relation
\begin{equation} \label{eq:pjbad}
	p^{(j)} = \left( \mcI - \mcT \right)^j p^{(0)} + \sum_{\ell=0}^{j-1} \sum_{i=0}^{\ell} \beta_{i\ell} \left( \mcI - \mcT \right)^{j-1-\ell} p^{(i)},
	\qquad
	j=0,1,\ldots.
\end{equation}
\end{theorem}
Although this relation can be used in combination with the binomial identity \eqref{eq:binom} to recursively compute $p^{(j)}$, this
approach is bound to result in numerical instabilities when the distance $d$ between the obstacles $\Omega_1$ and $\Omega_2$
is close to zero since, in this case, the asymptotic rate of convergence $\mathcal{R}_{2}$ is close to $1$. Concentrating for
instance on the term $\left( \mcI - \mcT \right)^j p^{(0)}$, this instability is apparent from the subtractive cancellations in binomial
identity \eqref{eq:binom} upon noting that $p^{(0)} = g$ and $\eta^{\ell+2} \sim \mathcal{R}_{2} \eta^{\ell} \sim \eta^{\ell}$
for $\ell \gg \log k$ by inequality \eqref{eq:etageomrate} and Theorem \ref{thm:2per}.

On the other hand, since $p^{(0)} = g$, a combined use of \eqref{eq:binom} and \eqref{eq:pjbad} clearly shows that the iterates
$p^{(j)}$ generated by the ORTHODIR algorithms can alternatively be computed through the following \emph{identification procedure}.
\begin{corollary} 
Each $p^{(j)}$ is a linear combination of $\eta^{0},\ldots,\eta^{j}$, say
\begin{equation} \label{eq:pj}
	p^{(j)} = \sum_{i=0}^{j} \gamma_{ij} \, \eta^{i},
\end{equation}
this allows for the computation of the next iterate as
\begin{align} 
	p^{(j+1)}
	= \left( \mcI - \mcT \right) p^{(j)} + \sum_{i=0}^{j} \beta_{ij} \, p^{(i)}
	&
	= \sum_{i=0}^{j} \gamma_{ij} \, \eta^{i}
	-\sum_{i=0}^{j} \gamma_{ij} \, \eta^{i+1}
	+\sum_{i=0}^{j} \beta_{ij} \, p^{(i)} \nonumber
	\\
	& 
	= \sum_{i=0}^{j+1} \gamma_{i,j+1} \, \eta^{i} \label{eq:pjgood}
\end{align} 
where the new coefficients $\gamma_{i,j+1}$ are easily computed by identification.
\end{corollary}

Note specifically that, since the phases of $\eta^i$ are known, identity \eqref{eq:pj} allows for a utilization
of the \emph{localized integration} scheme briefly summarized in \S4 in the evaluation of inner products
in steps 2.1 and 2.4 in the ORTHODIR iteration. On the other hand, the identification procedure
\eqref{eq:pjgood} provides a \emph{numerically stable} way of recursively computing $p^{(j)}$ as it clearly
eliminates subtractive cancellations arising from the use of binomial identity \eqref{eq:binom}.

 
\section{Preconditioning using Kirchhoff approximations}
\label{sec:Kirchhoff}

Although the novel Krylov subspace approach discussed in the previous section provides an effective
mechanism for the accelerated solution of multiple scattering problem \eqref{eq:OE}, this can be 
further improved if the operator equation \eqref{eq:OE} is properly preconditioned. Indeed, for an
appropriately defined operator $\mcK$ 
approximating the iteration operator $\mcT$, the preconditioned form of equation \eqref{eq:OE} reads 
\begin{equation} \label{eq:precond}
	\left( \mcI - \mcK \right)^{-1} \left( \mcI - \mcT \right) \eta = \left( \mcI - \mcK \right)^{-1} g.
\end{equation}
In this connection, we note the following useful alternative.
\begin{theorem}  \label{thm:altpre}
If the spectral radius $r(\mcK)$ of $\mcK$ is strictly less than $1$, then the preconditioned equation
\eqref{eq:precond} can be written alternatively as
\begin{equation} \label{eq:altpre}
	\left( \mcI - \sum_{\ell=0}^{\infty} \mcK^{\ell} \left( \mcT - \mcK \right) \right) \eta
	= \sum_{\ell=0}^{\infty} \mcK^{\ell} g.
\end{equation}
\end{theorem}
\begin{proof}
Since $r(\mcK) < 1$, we have the Neumann series representation \cite{Lebedev97} 
\begin{equation} \label{eq:NeumannK}
	\left( \mcI - \mcK \right)^{-1}
	= \sum_{\ell=0}^{\infty} \mcK^{\ell}.
\end{equation}
Use of \eqref{eq:NeumannK} in the identity
\[
	\left( \mcI - \mcK \right)^{-1} \left( \mcI - \mcT \right)
	= \mcI - \left( \mcI - \mcK \right)^{-1} \left( \mcT - \mcK \right)
\]
delivers the desired result.
\end{proof}

It is therefore natural to approximate the solution of \eqref{eq:OE} with the solution of the truncated equation
\begin{equation} \label{eq:trpre}
	\left( \mcI - \sum_{\ell=0}^{N} \mcK^{\ell} \left( \mcT - \mcK \right) \right) \eta
	= \sum_{\ell=0}^{M} \mcK^{\ell} g
\end{equation}
which we shall write as
\begin{equation} \label{eq:trpreshort}
	\mcA_{_{\mcK,N}} \, \eta = g_{_{\mcK.M}}.
\end{equation}
While equation \eqref{eq:trpreshort} displays the preconditioning strategy we shall utilize for the solution of
multiple scattering problem \eqref{eq:OE}, it is clearly amenable to a treatment by the Krylov subspace method
developed in the preceding section to further accelerate the solution of problem \eqref{eq:OE}.

As for the requirement that $\mcK$ has to approximate the iteration operator $\mcT$, we recall that
each application of $\mcT$ corresponds exactly to the evaluation of the surface current on each of
the obstacles $\Omega_1$ and $\Omega_2$ generated by the fields scattered from, respectively, 
$\Omega_2$ and $\Omega_1$ at the previous reflection as depicted by the identity
\[
	\begin{bmatrix}
		\eta_{1}^{m} \vspace{0.2cm} \\
		\eta_{2}^{m}
	\end{bmatrix}
	= \mcT
	\begin{bmatrix}
		\eta_{1}^{m-1} \vspace{0.2cm} \\
		\eta_{2}^{m-1}
	\end{bmatrix}
	= \begin{bmatrix}
		0 & \left( \mcI - \mcS_{11} \right)^{-1} \mcS_{12} \vspace{0.2cm} \\
		\left( \mcI - \mcS_{22} \right)^{-1} \mcS_{21} & 0
	\end{bmatrix}
	\begin{bmatrix}
		\eta_{1}^{m-1} \vspace{0.2cm} \\
		\eta_{2}^{m-1}
	\end{bmatrix}.
\]
It is therefore reasonable to define the operator $\mcK$ in the form
\[
	\mcK
	= \begin{bmatrix}
		0 & \mcK_{12} \\
		\mcK_{21} & 0
	\end{bmatrix}
\]
and require that $ \eta^{m}_{j} \approx \mcK_{jj'} \, \eta^{m-1}_{j'}$. Accordingly, the operators $\mcK_{jj'}$
must retain the phase information to preserve the frequency independent operation count while,
concurrently, providing a reasonable approximation to the slow densities to guarantee an accurate
preconditioning. This requirement can be satisfied only if the operators $\mcK_{jj'}$ are defined in
a \emph{dynamical} manner so as to respect the information associated with the iterates, and this
distinguishes our preconditioning strategy from classical approaches where the preconditioners are 
\emph{steady} by design. The most natural approach is to design the operators $\mcK_{jj'}$ so that they yield
the classical \emph{Kirchhoff approximations} as these preserve the phase information exactly and
approximate $\eta^{m,\, {\rm slow}}_{j}$ with the leading term in its asymptotic expansion. Concentrating on
two-dimensional settings, in this connection, a basic relation we exploited in \cite{EcevitReitich09} was
the observation that while, on the one hand, this term coincides with that of twice the normal derivative
of $u^{m-1}_{j'}$ in \eqref{eq:umj} on the illuminated region $\partial \Omega^{m}_{j} (IL)$, and on the
other hand, identity \eqref{eq:H10der} combined with asymptotic expansions of Hankel functions
\cite{AbramowitzStegun14} entails
\begin{equation} \label{eq:MF}
	\partial_{\nu(x)} G(x,y)
	\sim
	e^{ik \, |x-y|}
	\left( 
		e^{-i\pi/4}
		\sqrt{\dfrac{k}{2\pi \, |x-y|}} \
		\dfrac{x-y}{|x-y|} \cdot \nu(x)
	\right)
\end{equation}
so that use of \eqref{eq:MF} in \eqref{eq:umj} yields 
\begin{equation} \label{eq:RHS}
	2 \, \partial_{\nu(x)} u^{m-1}_{j'}(x)
	\sim 
	\int_{\partial \Omega_{j'}}
	\sqrt{\dfrac{k}{2\pi}} \,
	e^{ik \, (\varphi^{m-1}_{j'}(y) + |x-y|)-i\pi/4} \,
	\eta^{m-1, \, {\rm slow}}_{j'}(y) \,
	F(x,y) \,
	ds(y),
	\quad x \in \partial \Omega_{j},
\end{equation}
where
\begin{equation} \label{eq:defineF}
	F(x,y)
	= \dfrac{1}{\sqrt{|x-y|}} \
	\dfrac{x-y}{|x-y|} \cdot \nu(x).
\end{equation}
As for the oscillatory integral in \eqref{eq:RHS}, as we have shown in \cite{EcevitReitich09},
it is treatable through an appropriate use of \emph{stationary phase method} \cite{Fedoryuk71} which
states that the main contribution to an oscillatory integral comes from the stationary points of the phase.
\begin{lemma}[Stationary phase method]
Let $\psi \in C^{\infty}[a,b]$ be real valued, and let $h \in C^{\infty}_{0}[a,b]$.
Suppose that $t_{0}$ is the only stationary point of $\psi$ in $(a,b)$, $\psi''(t_{0}) \ne 0$,
and $\sigma = \operatorname{sign} \psi''(t_{0})$. Then there exists a constant $C$ such that,
for all $k > 1$,
\[
	\left|
		\int_{a}^{b} e^{ik\psi(t)} \, h(t) \, dt
		- e^{ik\psi(t_{0})+i\pi\sigma/4} \,
		h(t_0) \,
		\sqrt{ \dfrac{2\pi}{k \, |\psi''{t_0}|}} \,
	\right|
	\le C \, k^{-1} \, \Vert h \Vert_{C^{2}[a,b]}.
\]
\end{lemma}
\noindent
Indeed, it turns out \cite{EcevitReitich09} that the combined phase function
\[
	\varphi^{m}_{jj'}(x,y) = \varphi^{m-1}_{j'}(y) + |x-y|
\]
has two stationary points, one in the shadow region $\partial \Omega^{m-1}_{j'}(SR)$ with a contribution of
$\mathcal{O}(k^{-n})$ (for all $n \in \mathbb{N}$) due to rapid decay of the amplitude $\eta^{m-1, \, {\rm slow}}_{j'}$,
and another one in the illuminated region $\partial \Omega^{m-1}_{j'}(IL)$ given by $y(x) = \mathcal{X}^{m}_{m-1}(x)$
(at which the combined phase has a positive ``second derivative'') whose contribution agrees, to leading order,
with that given by stationary phase evaluation of the integral in \eqref{eq:RHS}. While this discussion clarifies how
\emph{Kirchhoff operators} $\mathcal{K}_{jj'}$ must be designed so that they yield the leading terms in the asymptotic
expansions of $\eta^{m}_{j}$ on the illuminated regions $\partial \Omega^{m}_{j}(IL)$ at each iteration, the rapid decay
of $\eta^{m}_{j}$ in the shadow region $\partial \Omega^{m}_{j}(SR)$, in  turn, provides the motivation that $\mathcal{K}_{jj'}$
must simply approximate $\eta^{m}_{j}$ by zero in these regions. 
Being aware of these, we use $\gamma_j(t_j) = (\gamma_j^1(t_j),\gamma_j^2(t_j))$ to denote the arc length
parametrezation of $\partial \Omega_{j}$ (in the counterclockwise orientation) with period $L_j$ ($j=1,2$)
so that, for each $x_j \in \partial \Omega_j$, $t_j$ is the unique point in $[0,L_j)$ with $\gamma_j(t_j) = x_j$,
and define the \emph{Kirchhoff operators} $\mathcal{K}_{jj'}$ as follows.

\begin{definition}
For a smooth \emph{phase} $\phi_{j'} : \partial \Omega_{j'} \to \mathbb{R}$ having the property that,
for each $x_j \in \partial \Omega_{j}$, the function $\phi_{jj'}: \partial \Omega_{j} \times \partial \Omega_{j'} \to \mathbb{R}$
given by
\[
	\phi_{jj'}(x_j,x_{j'}) = \phi_{j'}(x_{j'}) + |x_j-x_{j'}|
\]
has a unique stationary point $y_{j'} = x_{j'}(x_j) \in \partial \Omega_{j'}$ such that $[x_{j},y_{j'}] \cap \partial \Omega_{j'} = y_{j'}$,
define the \emph{transformed phase} $\phi_{j} : \partial \Omega_j \to \mathbb{R}$ by setting
\[
	\phi_j (x_j) = \phi_{jj'}(x_j,y_{j'}).
\]
Assume further $\phi_{jj'} (t_j,t_{j'}) = \phi_{jj'}(x_j,x_{j'})$ has $\partial^2_{t_{j'}} \phi_{jj'} (t_j,\tau_{j'}) > 0$ 
at $\tau_{j'} = \gamma_{j'}^{-1}(y_{j'})$ and for a given \emph{amplitude} $A_{j'} : \partial \Omega_{j'} \to \mathbb{C}$,
define the \emph{transformed amplitude} $A_j : \partial \Omega_j \to \mathbb{C}$ by setting
\[
	A_j(x_j) =
	\left\{
		\begin{array}{cl}
			B_j(x_j), & \mbox{ if  } \ [x_j,y_{j'}] \cap \partial \Omega_{j} = \{ x_j \},
			\\
			\\
			0, & \mbox{ otherwise},
		\end{array}
	\right.
\]
where, with the function $F$ as defined in \eqref{eq:defineF},
\[
	B_j (x_j)
	= A_{j'}(y_{j'}) \, F(x_j,y_{j'}) \, ( \partial^2_{t_{j'}} \phi_{jj'} (t_j,\tau_{j'}))^{-1/2}.
\]
Finally, define the \emph{Kirchhoff operator} $\mathcal{K}_{jj'}$ by setting
\begin{equation} \label{eq:Kirchhoffactual}
	\mathcal{K}_{jj'}(\phi_{j'},A_{j'})
	= (\phi_{j},A_{j}).
\end{equation}
\end{definition}

We abbreviate identity \eqref{eq:Kirchhoffactual} as
\begin{equation} \label{eq:Kshort}
	\mathcal{K}_{jj'} (e^{ik\, \phi_{j'}} \, A_{j'} ) = e^{ik\, \phi_{j}} \, A_{j},
\end{equation}
and extend $\mathcal{K}_{jj'}$ by linearity so that
\begin{equation} \label{eq:Klinear}
	\mathcal{K}_{jj'} (\sum_{\ell=0}^{N} e^{ik\, \phi^{\ell}_{j'}} \, A^{\ell}_{j'} )
	= \sum_{\ell=0}^{N} \mathcal{K}_{jj'} (e^{ik\, \phi^{\ell}_{j'}} \, A^{\ell}_{j'}).
\end{equation}

In connection with the requirement that the operators $\mcK_{jj'}$ must retain the phase information
exactly while providing a reasonable approximation to the slow densities, we note that
\[
	\mathcal{K}_{jj'} (e^{ik\, \varphi^{m-1}_{j'}} \, \eta^{m-1, \, {\rm slow}}_{j'} )(x)
	= e^{ik \, \varphi^{m}_{j}(x)} \,
	\lambda^{m, \, {\rm slow}}_{j}(x),
	\quad
	x \in \partial \Omega_{j},
\]
where, with $F$ as given in \eqref{eq:defineF},
\[
	\lambda^{m, \, {\rm slow}}_{j}(x) =
	\left\{
		\begin{array}{cl}
			\eta^{m-1, \, {\rm slow}}_{j'}(\mathcal{X}^{m}_{m-1}(x)) \,
			F(x,\mathcal{X}^{m}_{m-1}(x)) \,
			( \partial^2_{t_{j'}} \varphi^{m}_{jj'} (t_j,\tau_{j'}))^{-1/2}, &
			x \in \partial \Omega^{m}_{j}(IL),
			\\
			\\
			0, &
			\text{ otherwise},
		\end{array}
	\right.
\]
so that $\mathcal{K}_{jj'}$ preserves the phase information, and the leading term in the asymptotic
expansion of $\lambda^{m, \, {\rm slow}}_{j}$ agrees with that of $\eta^{m, \, {\rm slow}}_{j}$ in the illuminated
region $\partial \Omega^{m}_{j}(IL)$ as desired (cf. \cite[Theorems 3.3 and 3.4]{EcevitReitich09}).

As for the alternative form \eqref{eq:altpre} of the preconditioned equation \eqref{eq:precond}, we have the
following result. 

\begin{theorem}
Suppose that the obstacles $\Omega_1$ and $\Omega_2$ satisfy the no-occlusion condition with
respect to the direction of incidence $\alpha$. Considering any given function $h \in C(\partial \Omega)$ as
$h(x) = e^{ik \, \alpha \cdot x} \, h_{0}(x)$, the series
\[
	\sum_{\ell=0}^{\infty} \Vert \mcK^{\ell} h \Vert_{\infty}
\]
converges for all $k >0$.
\end{theorem}
\begin{proof}
The same technique used to prove \cite[Theorem 4.1]{EcevitReitich09} entails the existence of constants $C = C(\Omega,\alpha) >0$
and $\delta = \delta(\Omega,\alpha) \in (0,1)$ such that, for all $\ell \in \mathbb{Z}_+$,
\[
	\Vert \mcK^{\ell+2} h - \mcR_2 \mcK^{\ell} h \Vert_{\infty}
	\le \delta^{\ell} \left( \min \left\{ 2, \exp \left( Ck \, \delta^{\ell} \right)-1 \right\} \delta^2 + C \delta^{\ell-2} \right) \Vert h \Vert_{\infty}
\]
which yields
\begin{align*}
	\Vert \mcK^{\ell+2} h - \mcR_2 \mcK^{\ell} h \Vert_{\infty}
	& \le \delta^{\ell} \left( 2 \delta^{2} + C \delta^{\ell-2} \right) \Vert h \Vert_{\infty}.
\end{align*}
Using $C = C(\Omega,\alpha)$ to denote a positive constant whose value may be different at each appearance in what follows,
this inequality clearly implies
\[
	\Vert \mcK^{\ell+2} h - \mcR_2 \mcK^{\ell} h \Vert_{\infty}
	\le C \, \delta^{\ell} \Vert h \Vert_{\infty}.
\]
Since $|\mcR_2| < 1$, choosing $\delta$ larger, if necessary, we may assume that $\delta^2 \in (|\mcR_2|,1)$. 
In this case, the preceding inequality yields, for $\ell \in \mathbb{Z}_{+}$ and $m=0,1$,
\begin{align*}
	\Vert \mcK^{m+2\ell} h \Vert_{\infty}
	& \le \Vert \mcR_2^{\ell} \, \mcK^{m} h \Vert_{\infty}
	+\sum_{j=0}^{\ell-1} \Vert \mcR_2^{\ell-(j+1)} \mcK^{m+2(j+1)} h - \mcR_2^{\ell-j}\mcK^{m+2j} h \Vert_{\infty}
	\\
	& = |\mcR_2|^{\ell} \, \Vert \mcK^{m} h \Vert_{\infty}
	+\sum_{j=0}^{\ell-1} |\mcR_2|^{\ell-(j+1)} \Vert \mcK^{m+2(j+1)} h - \mcR_2 \mcK^{m+2j} h \Vert_{\infty}
	\\
	& \le |\mcR_2|^{\ell} \, \Vert \mcK^{m} h \Vert_{\infty}
	+ C \sum_{j=0}^{\ell-1} |\mcR_2|^{\ell-(j+1)} \delta^{m+2j} \Vert h \Vert_{\infty}
	\\
	& = |\mcR_2|^{\ell} \, \Vert \mcK^{m} h \Vert_{\infty}
	+C \, \delta^{m} \dfrac{|\mcR_2|^{\ell}-\delta^{2\ell}}{|\mcR_2|- \delta^{2}} \Vert h \Vert_{\infty}
	\\
	& \le \delta^{2\ell} \, \Vert \mcK^{m} h \Vert_{\infty}
	+C \, \delta^{m+2\ell} \Vert h \Vert_{\infty}.
\end{align*}
Since, we clearly have $\Vert \mcK^{m} h \Vert_{\infty} \le C \Vert h \Vert_{\infty}$ for $m=0,1$, we conclude
\[
	\Vert \mcK^{m+2\ell} h \Vert_{\infty}
	\le C \left( \delta^{2\ell} + \delta^{m+2\ell} \right)\Vert h \Vert_{\infty}
	\le C \, \delta^{m+2\ell} \, \Vert h \Vert_{\infty},
\]
and this gives, for all $\ell \in \mathbb{Z}_+$,
\begin{equation} \label{eq:estKl}
	\Vert \mcK^{\ell} h \Vert_{\infty}
	\le  C \, \delta^{\ell} \, \Vert h \Vert_{\infty}. 
\end{equation}
Thus the result.
\end{proof}

\begin{remark}
Considering $\mcK$ as an operator $\mcK : C(\partial \Omega) \to C(\partial \Omega)$, inequality
\eqref{eq:estKl} implies
\[
	r (\mcK)
	= \lim_{\ell \to \infty} \Vert \mcK^{\ell} \Vert_{\infty}^{1/\ell}
	\le \delta <1
\]
for the spectral radius of $\mcK$, and this explains the sense in which identity \eqref{eq:altpre}
in Theorem \ref{thm:altpre} holds.
\end{remark}

In connection with the application of the ORTHODIR iteration to the preconditioned equation \eqref{eq:trpre},
setting $\varphi^{m} = [\varphi^{m}_1, \varphi^{m}_2]^{t}$ and using $\mu^{\ell}$ ($\ell = 0, 1, \ldots$) to denote
generic functions defined on $\partial \Omega$ which may be different from line to line, we thus see through
equations \eqref{eq:Kshort}-\eqref{eq:Klinear} that
\[
	p^{(0)} = g_{_{\mcK,M}}
	= \sum_{\ell=0}^{N} \mathcal{K}^{\ell}g
	= \sum_{\ell=0}^{M} \mathcal{K}^{\ell} (e^{ik \, \varphi^{0}} \eta^{0, \, {\rm slow}})
\]
is of the form
\begin{equation} \label{eq:p0}
	p^{(0)} = \sum_{\ell=0}^{M} e^{ik \, \varphi^{\ell}} \mu^{\ell}.
\end{equation}
More generally, we have the following result.

\begin{theorem}
For $j=0,1,2\ldots$, the ORTHODIR iterates $p^{(j)}$ are of the form
\begin{equation} \label{eq:pj0}
	p^{(j)} = \sum_{\ell=0}^{M+j(N+1)} e^{ik \, \varphi^{\ell}} \mu^{\ell}.
\end{equation}
\end{theorem}

\begin{proof}
This follows by a straightforward induction based on equations \eqref{eq:Kshort}-\eqref{eq:Klinear}, \eqref{eq:p0}
and the recursion
\begin{equation} \label{eq:pj1}
	p^{(j+1)}
	= \mcA_{_{\mcK,N}} p^{(j)} + \sum_{i=0}^{j} \beta_{ij} \, p^{(i)}
	= \left( \mcI - \sum_{\ell=0}^{N} \mcK^{\ell} \left( \mcT - \mcK \right) \right) p^{(j)} + \sum_{i=0}^{j} \beta_{ij} \, p^{(i)}.
\end{equation}
\end{proof}

The main point behind this theorem is that use of \eqref{eq:pj0} in \eqref{eq:pj1} clearly allows for an application
of the aforementioned localized integration scheme in connection with the execution of the operator $\mcT$ in \eqref{eq:pj1}.
Moreover, it is clear that each realization of the Kirchhoff operator $\mcK$ is frequency
independent. Consequently, the preconditioned equation \eqref{eq:trpre} is amenable to a treatment by the Krylov
subspace method described in \S5 to obtain even more accelerated solutions of the multiple scattering problem
\eqref{eq:OE} while still retaining the frequency independent operation count if desired.


\section{Numerical implementations}
\label{sec:Numerical}

Here we present numerical examples that display the
benefits of our Krylov subspace approach as well as its
preconditioning through use of Kirchhoff approximations.
To this end, we have designed two different test configurations
(see Fig.~\ref{fig:configurations}).
First we have considered two circles illuminated by a plane-wave
incidence coming in from the left with wavenumber $k=200$.
While the radii of the circles are $1$ and $1.5$, they are
centered at the origin and $(0.9625,-2.6444)$ respectively.
Second we have treated a configuration consisting of two
parallel elliptical obstacles with centers at $(0,0)$ and
$(0,-4.5)$, and major/minor axes $10/1$ and $7/2$. The
illumination is provided by a plane wave with direction
along the major axes and wavenumber $k = 40$.

\begin{figure}
\begin{center}
\hbox{
	\hspace{1.6cm}
	\subfigure[Circles]
		{\includegraphics*[height=1.8in,viewport=190 110 650 490,clip] {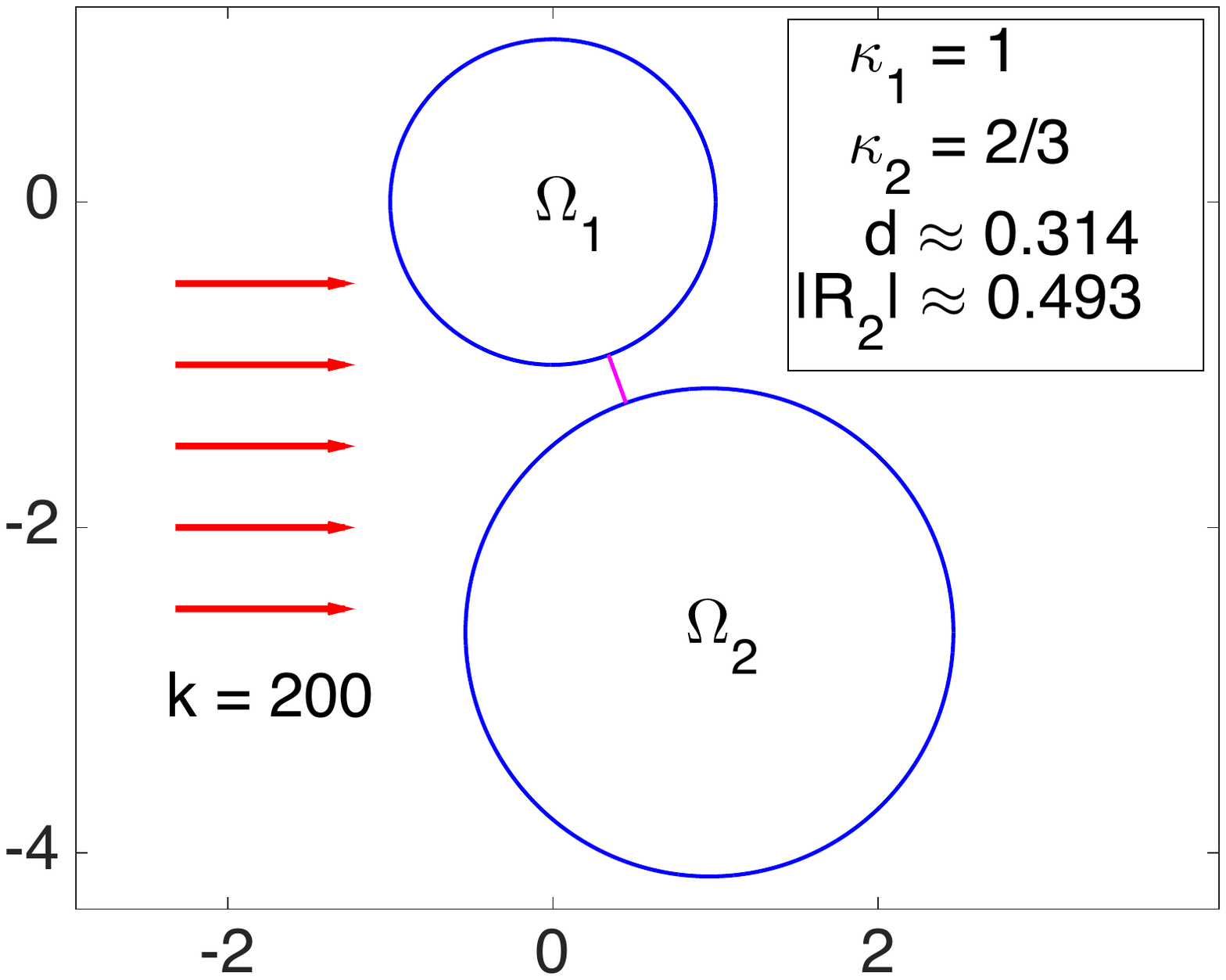}}
	\hspace{1.8cm}		
	\subfigure[Ellipses]
	{\includegraphics*[height=1.8in,viewport=190 110 650 490,clip] {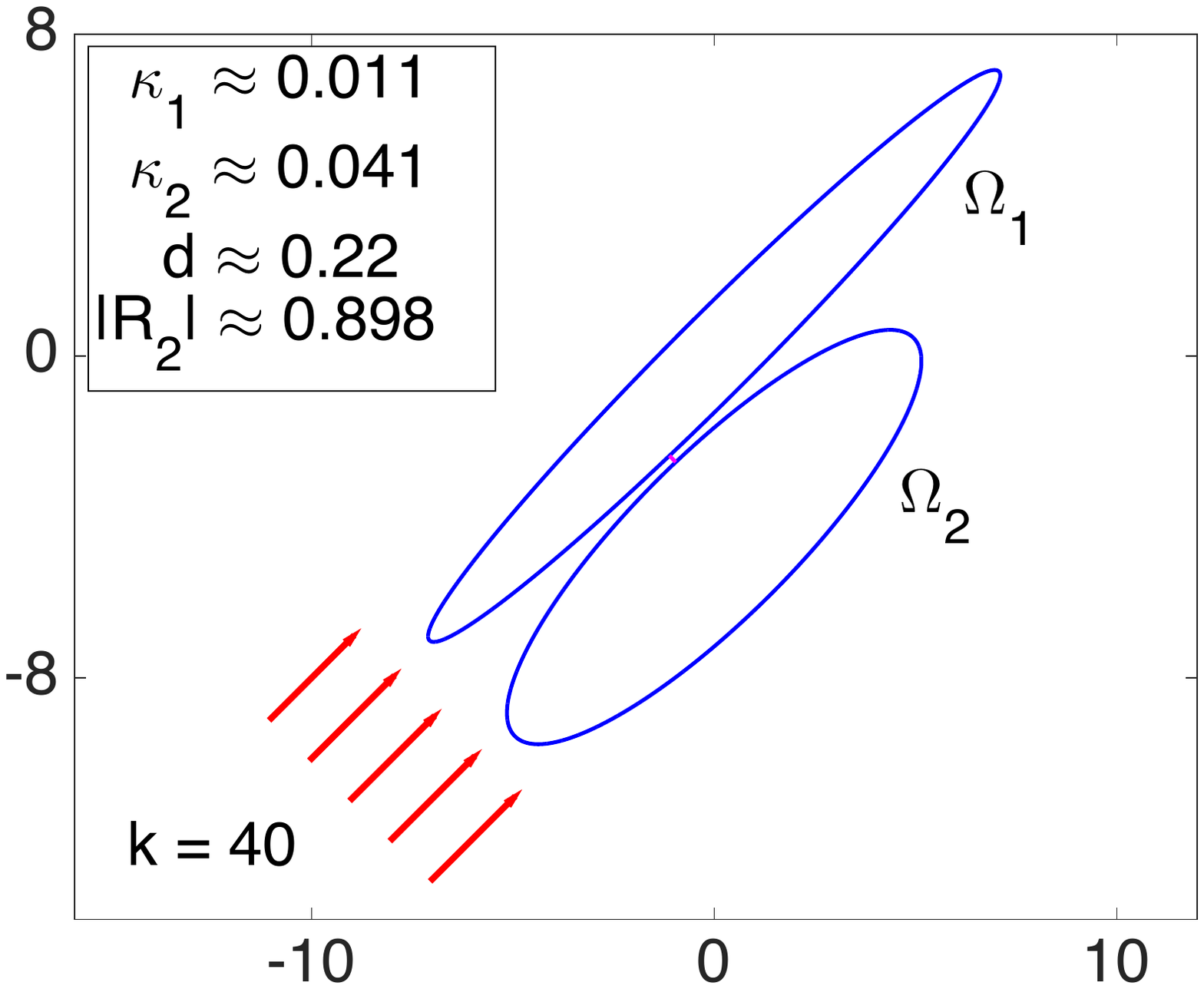}}
}
\end{center}\vspace{0cm}
\caption{Multiple scattering configurations.}
\label{fig:configurations}
\end{figure}

\begin{figure}[pb]
\begin{center}
\hbox{
	\subfigure[Circles]
		{\includegraphics*[width=3.0in,viewport=30 90 770 490,clip] {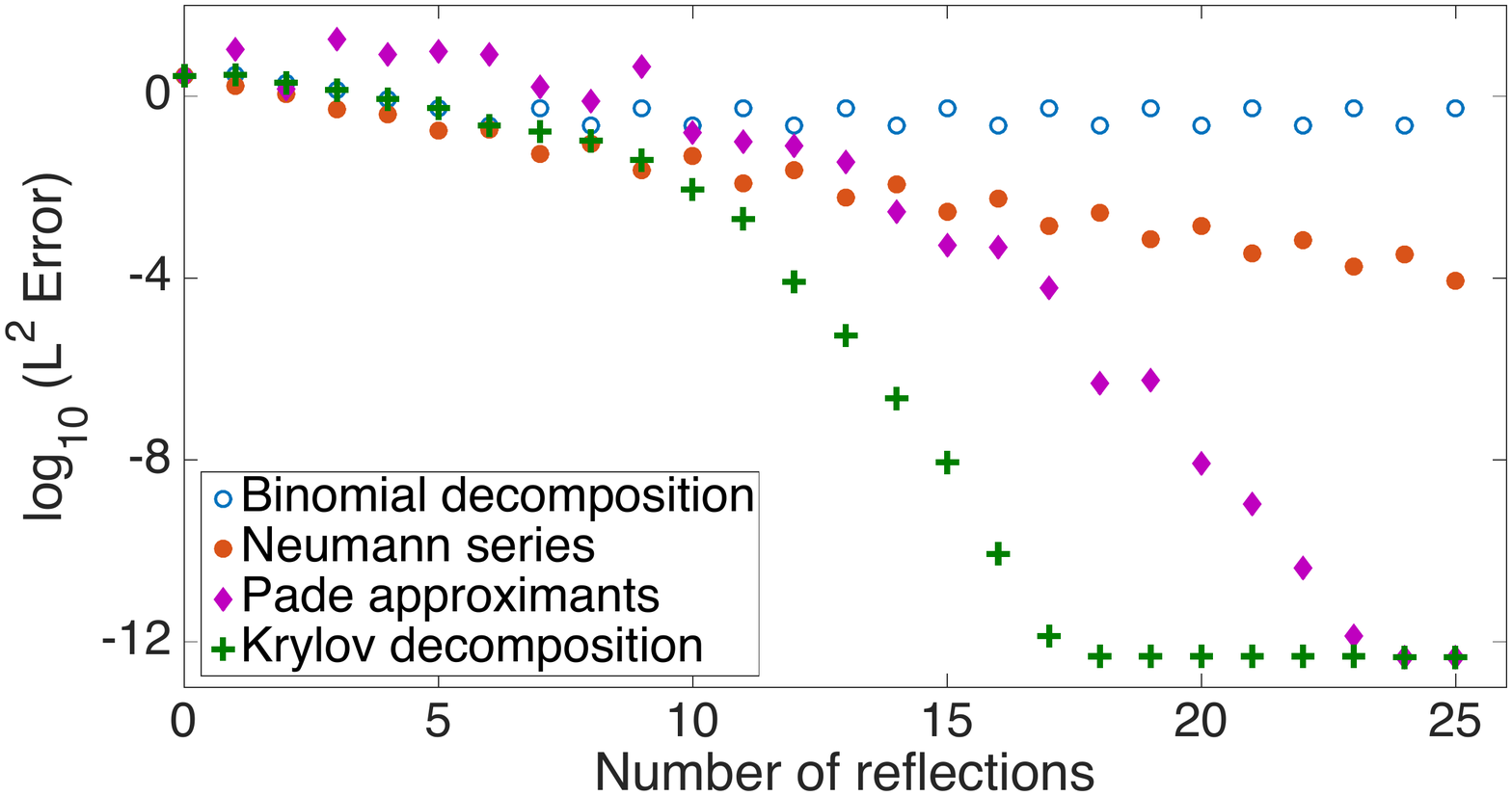}}
	\subfigure[Ellipses]
		{\includegraphics*[width=3.0in,viewport=30 90 770 490,clip] {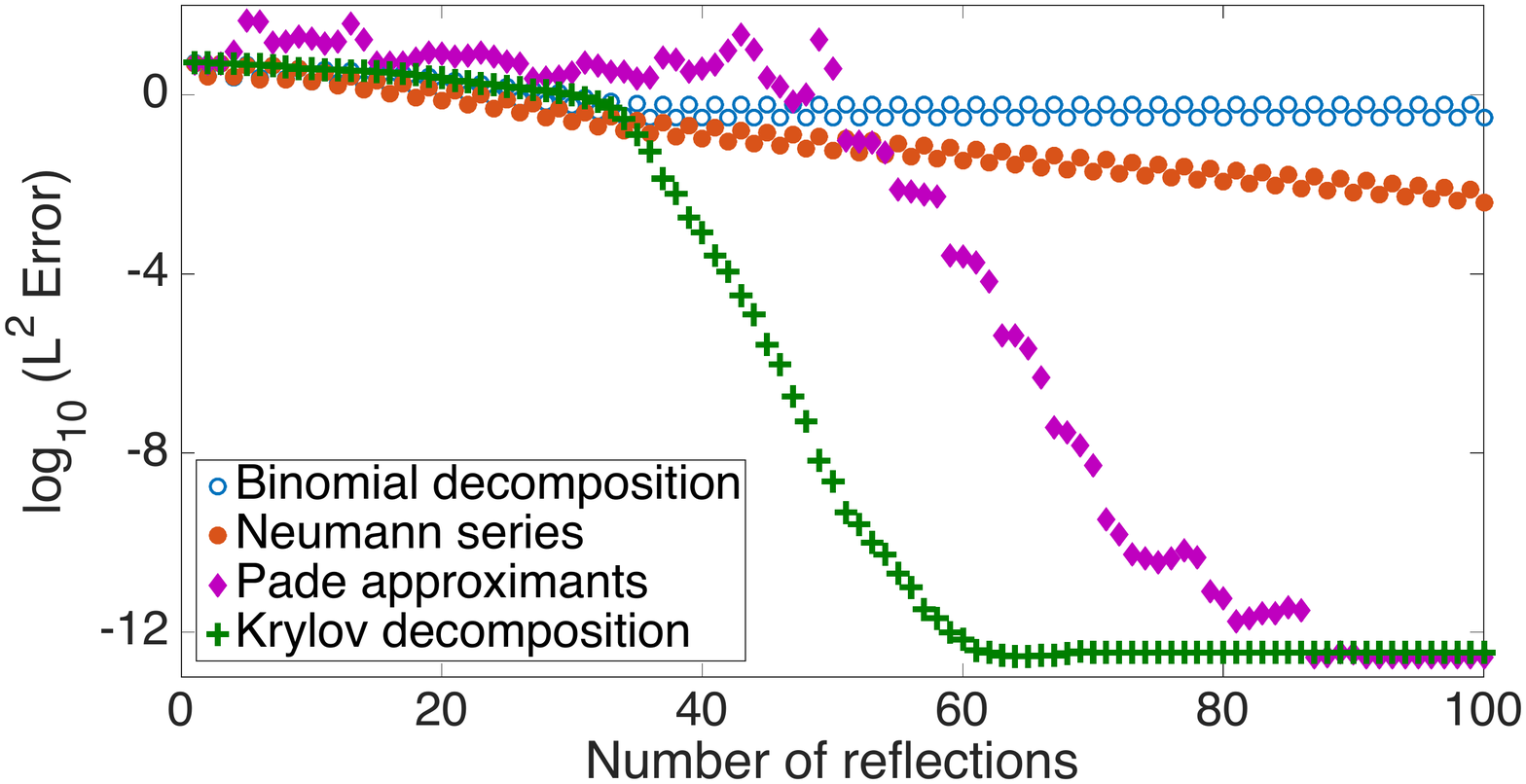}}
}
\end{center}\vspace{-1cm}
\caption{Acceleration provided by the Krylov subspace method.}
\label{fig:Krylov}
\end{figure}

Figure \ref{fig:Krylov} provides a comparison of (a) the Neumann series,
(b) the Pad\'{e} approximants, (c) the Krylov subspace method based on
a combined use of binomial formula \eqref{eq:binom} and identity
\eqref{eq:pjbad}, and (d) the alternative implementation of the latter
based on decomposition \eqref{eq:pj} leading to equation
\eqref{eq:pjgood}. More precisely, Figure \ref{fig:Krylov} depicts the
number of reflections versus the logarithmic $L^2$ error
\[
	\log_{10} \Vert \eta - \hat{\eta} \Vert_{2}
\]
between the exact solution $\eta$ and the approximations $\hat{\eta}$
obtained by the four aforementioned schemes. In both cases, the reference 
solution $\eta$ is computed using an integral solver with sufficiently many
disretization points to guarantee $14$ digits of accuracy. 
As we anticipated,
combined use of binomial formula \eqref{eq:binom} and identity
\eqref{eq:pjbad} suffers from subtractive cancellations and fails to
approximate the solution as the number of reflections increases.
The implementation of Krylov subspace method based on decomposition
\eqref{eq:pj} and resulting equation  \eqref{eq:pjgood} clearly resolves this
issue. Furthermore, when compared with the Pad\'{e} approximants 
considered in \cite{BrunoEtAl05}, approximations provided by this
alternative implementation of the Krylov subspace method are more
stable and give better accuracy at each iteration. Incidentally, note
specifically that a direct use of Neumann series would require about
$77/522$ iterations to obtain $12$ digits of accuracy for circular/elliptical
configurations in Figure~\ref{fig:configurations}, and thus our Krylov subspace
approach provides savings of $78\%/87\%$ in the required number of
reflections.

Finally, in Figure~\ref{fig:Kirchhoff}, we display a comparison of (a) the
Neumann series, (b) the stable implementation of our Krylov subspace
approach based on decomposition \eqref{eq:pj} and equation
\eqref{eq:pjgood}, and (c) Kirchhoff preconditioning of the latter. Note
precisely that (c) is based on the Krylov subspace iterations (described in
\S\ref{sec:Krylov}) applied to the truncated version \eqref{eq:trpreshort} of
preconditioned form \eqref{eq:altpre} of the multiple scattering problem
\eqref{eq:OE} utilizing the Kirchhoff operator $\mcK$. In our implementations
we have taken $N=M$ in equation \eqref{eq:trpreshort} and used $N= 12/40$
for the circular/elliptical configurations in Figure~\ref{fig:configurations}.
As depicted in Figure~\ref{fig:Kirchhoff}, in both cases only three ORTHODIR
iterations are sufficient to obtain $3-$digits of accuracy which would
require 20/100 iterations if Neumann series is directly used. The fact that the
error does not attain the machine precision is due to the truncation of the series
used to compute the preconditioner ($N=12/40$). Obviously inclusion of more
terms yields better accuracy but at the expense of slightly more expansive
numerics.

\begin{figure}[pt]
\begin{center}
\hbox{
	\subfigure[Circles]
		{\includegraphics*[width=3.0in,viewport=40 90 770 490,clip] {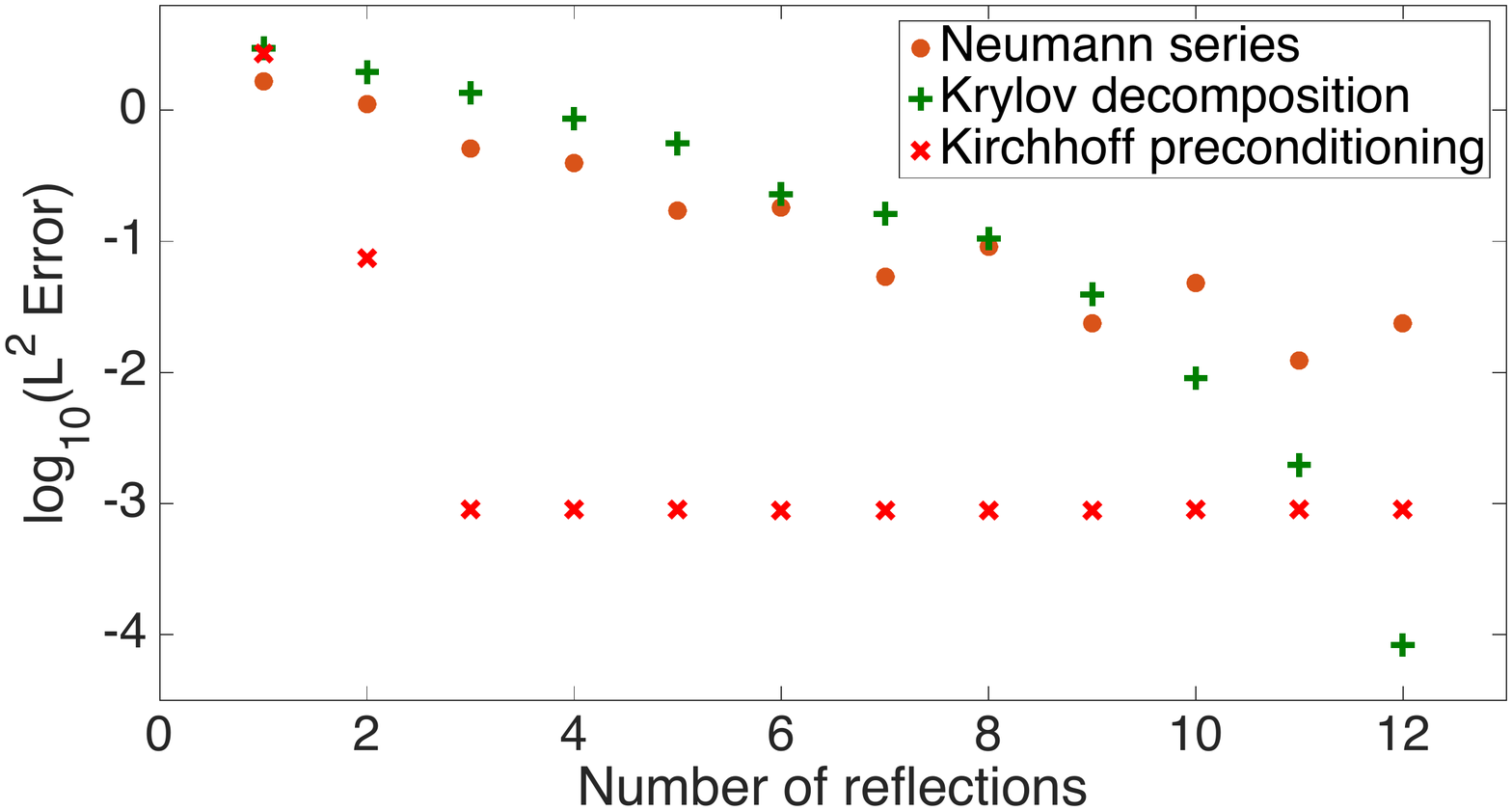}}
	\subfigure[Ellipses]
		{\includegraphics*[width=3.0in,viewport=40 90 770 490,clip] {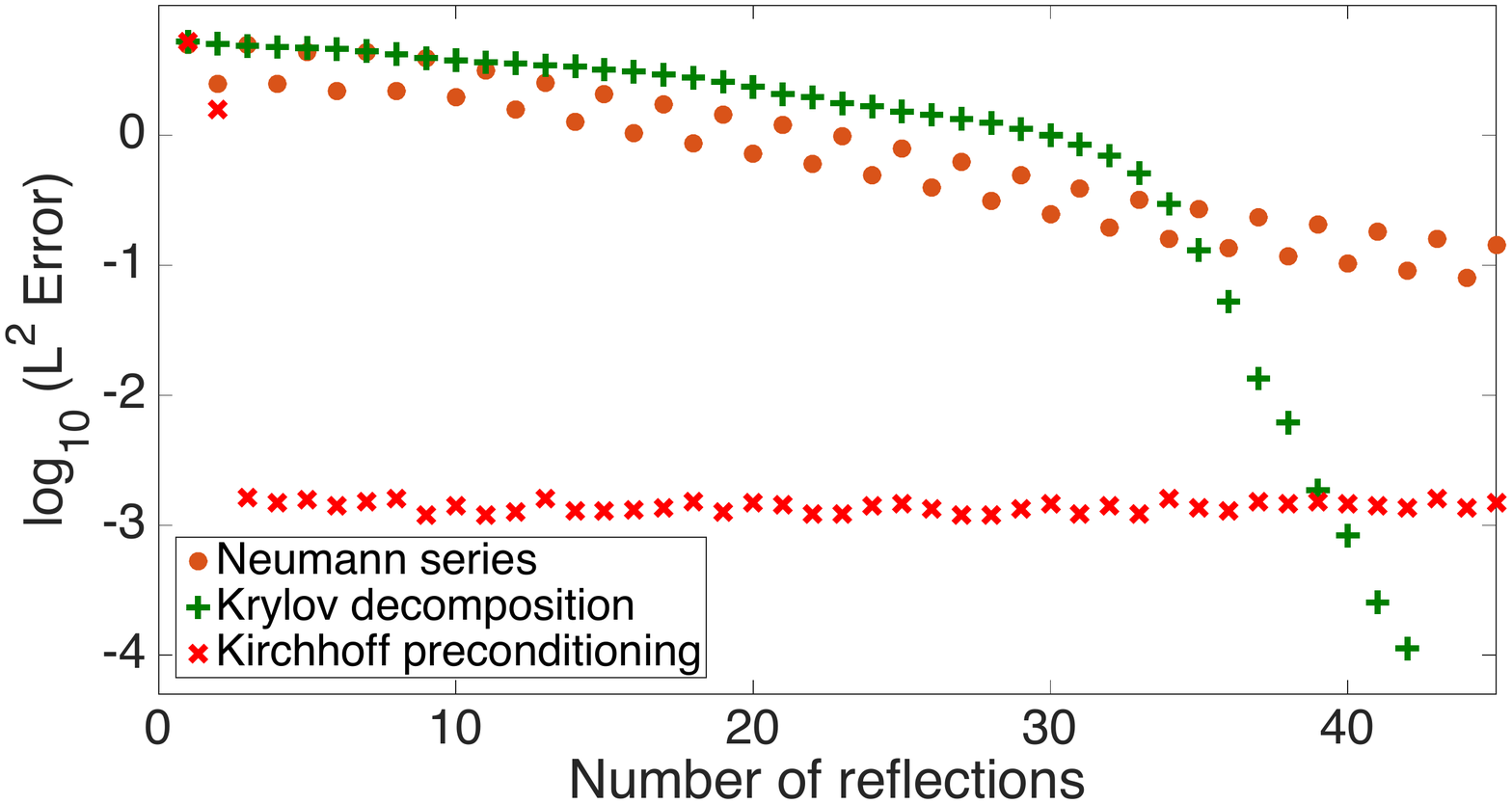}}
}
\end{center}\vspace{-1cm}
\caption{Preconditioning through Kirchhoff approximations.}
\label{fig:Kirchhoff}
\end{figure}


\section{Conclusion}
\label{sec:Conclusions}

We have developed an acceleration strategy for the solution of multiple scattering
problems based on a novel and effective use of Krylov subspace methods that retains the
phase information and provides significant savings in computational times. Further,
we have coupled this approach with an original preconditioning strategy based upon Kirchhoff
approximations that  greatly reduces the number of iterations needed to obtain a
prescribed accuracy. In the forthcoming work, we will extend this numerical algorithm
for configurations of more than two obstacles. Indeed, our new Krylov method can be easily
applied to this kind of configurations without adding any additional computational cost.
On the other hand, although the Kirchhoff preconditioner greatly enhances the convergence
of the Krylov subspace method for two obstacles, its utilization for several obstacles requires
some numerical optimization.

\section{Acknowledgments}

Y. Boubendir gratefully acknowledges support from NSF through grant No. DMS-1319720.


\begin{thebibliography}{99}
\bibitem{AbboudEtAl94} Abboud, T., N\'ed\'elec, J.-C., Zhou, B.:
M\'ethode des \'equations int\'egrales pour les hautes fr\'equences, 
C.R. Acad. Sci. Paris \textbf{318} (1994), 165--170.

\bibitem{AbboudEtAl95} Abboud, T., N\'ed\'elec, J.-C., Zhou, B.:
Improvement of the integral equation method for high frequency problems,
in Mathematical and Numerical Aspects of Wave Propagation: Mandelieu-La Napoule, SIAM,
(1995), 178--187.

\bibitem{AbramowitzStegun14}
Abramowitz, M., Stegun, I.:
\emph{Handbook of Mathematical Functions with Formulas, Graphs, and Mathematical Tables}
(Martino Fine Books, 2014).

\bibitem{AminiProfit03}
Amini, S., Profit, A.:
Multi-level fast multipole solution of the scattering problem,
Engineering Analysis with Boundary Elements \textbf{27} (2003), 547--564.

\bibitem{AnandEtAl10}
Anand, A., Boubendir, Y., Ecevit, F., Reitich, F.:
Analysis of multiple scattering iterations for high-frequency scattering problems. II. The three-dimensional scalar case,
Numer. Math. \textbf{114}(3) (2010), 373--427.

\bibitem{Antoine08}
Antoine, X.:
Advances in the on-surface radiation condition method: Theory, numerics and applications,
in F. Magoules (Ed.), Comput. Meth. for Acoustics Problems, Saxe-Coburg Publ. Stirlingshire, UK (2008), 207--232.

\bibitem{Balabane04}
Balabane, M.:
Boundary decomposition for Helmholtz and Maxwell equations. I. Disjoint sub-scatterers,
Asymptot. Anal. \textbf{38}(1) (2004), 1--10.

\bibitem{BanjaiHackbusch05}
Banjai, L., Hackbusch, W.:
Hierarchical matrix techniques for low- and high-frequency Helmholtz problems,
IMA J. Numer. Anal. \textbf{28}(1) (2008), 46--79.

\bibitem{Boffi10}
Boffi, D.:
Finite element approximation of eigenvalue problems,
Acta Numer. \textbf{19} (2010), 1--120.

\bibitem{BrunoGeuzaine07}
Bruno, O.P., Geuzaine, C.A.:
An $\mathcal{O}(1)$ integration scheme for three-dimensional surface scattering problems,
J. Comput. Appl. Math. \textbf{204} (2007), 463--476.

\bibitem{BrunoEtAl13}
Bruno, O.P., Dom\'inguez, V., Sayas, F.-J.:
Convergence analysis of a high-order Nystr\"om integral-equation method for surface scattering problems,
Numer. Math. \textbf{124}(4) (2013), 603--645.

\bibitem{BrunoKunyansky01}
Bruno, O.P., Kunyansky, L.A.:
A fast, high-order algorithm for the solution of surface scattering problems: basic implementation, tests and applications,
J. Comput. Phys. \textbf{169} (2001), 80--110.

\bibitem{BrunoEtAl04}
Bruno, O.P., Geuzaine, C.A., Monroe, J.A., Reitich F.:
Prescribed error tolerances within fixed computational times for
scattering problems of arbitrarily high frequency: the convex case,
Phil. Trans. Roy. Soc. London \textbf{362} (2004), 629--645.

\bibitem{BrunoEtAl05} Bruno, O.P., Geuzaine, C.A., Reitich, F.: 
On the $\mathcal{O}(1)$ solution of multiple-scattering problems,
IEEE Trans. Magn. \textbf{41} (2005), 1488--1491.





\bibitem{Chandler-WildeMonk08}
Chandler-Wilde, S.N., Monk, P.:
Wave-number-explicit bounds in time-harmonic scattering,
SIAM J. on Math. Anal. \textbf{39}(5) (2008), 1428--1455.

\bibitem{ColtonKress92}
Colton, D., Kress, R.:
\textit{Inverse Acoustic and Electromagnetic Scattering Theory}
(Springer-Verlag, Berlin, 1992).

\bibitem{DaviesEtAl09}
Davies, R. W., Morgan, K., Hassan, O.:
A high order hybrid finite element method applied to the solution of electromagnetic wave scattering problems in the time domain,
Comput. Mech. \textbf{44}(3) (2009), 321--331.

\bibitem{DominguezEtAl07}
Dom\'inguez, V., Graham, I.G., Smyshlyaev, V.P.:
A hybrid numerical-asymptotic boundary integral method for high-frequency acoustic scattering,
Numer. Math. \textbf{106}(3) (2007), 471--510.

\bibitem{EcevitOzen16}
Ecevit, F., \"{O}zen, H.\c{C}.:
Frequency-adapted Galerkin boundary element methods for convex scattering problems,
Numer. Math. (2016), DOI: 10.1007/s00211-016-0800-7.

\bibitem{EcevitReitich09}
Ecevit, F., Reitich, F.:
Analysis of multiple scattering iterations for high-frequency scattering problems. I. The two-dimensional case,
Numer. Math. \textbf{114}(2) (2009), 271--354.

\bibitem{EnquistMajda} Engquist, B., Majda, A.:
Absorbing boundary conditions for the numerical simulation of waves,
Math. Comp. \textbf{31}(139) (1977), 629--651.


\bibitem{Fedoryuk71} Fedoryuk, M.V.:
The stationary phase method and pseudodifferential operators,
Russian Math Surv. \textbf{26}(1) (1971), 65--115. 

\bibitem {GaneshHawkins11}
Ganesh, M., Hawkins, S.:
A fully discrete Galerkin method for high frequency exterior acoustic scattering in three dimensions,
J. Comput. Phys. \textbf{230} (2011), 104--125.

\bibitem{Giladi07}
Giladi, E.:
An asymptotically derived boundary element method for the Helmholtz equation in high frequencies,
J. Comput. Appl. Math. \textbf{198} (2007), 52--74.

\bibitem{Givoli04}
Givoli, D.:
High-order local non-reflecting boundary conditions: a review,
Wave Motion \textbf{39}(4) (2004), 319--326.

\bibitem{GroteKirsch07}
Grote, M.J., Kirsch, C.:
Nonreflecting boundary conditions for time-dependent multiple scattering,
J. Comp. Phys. \textbf{221}(1) (2007), 41--62.

\bibitem{GroteSim11}
Grote, M., Sim, I.:
Local nonreflecting boundary condition for time-dependent multiple scattering,
J. Comput. Phys. \textbf{230} (2011), 3135--3154.


\bibitem{HesthavenWarburton04}
Hesthaven, J.S., Warburton, T.:
High-order accurate methods for time-domain electromagnetics, 
CMES Comput. Model. Eng. Sci. \textbf{5} (2004), 395--408.

\bibitem{Hormander66} H\"ormander, L.:
Pseudo-differential operators and hypoelliptic equations. Singular integrals
(Proc. Sympos. Pure Math., Vol. X, Chicago, Ill., 1966), 138--183;
Amer. Math. Soc., Providence, R.I. (1967).

\bibitem{Hormander71c} H\"ormander, L.:
Fourier integral operators. I,
Acta Math. \textbf{127} (1971), no. 1--2, 79--183.

\bibitem{HuybrechsVandewalle07}
Huybrechs, D., Vandewalle, S.:
A sparse discretization for integral equation formulations of high frequency scattering problems,
SIAM J. Sci. Comput. \textbf{29}(6) (2007), 2305--2328.

\bibitem{Keller62} Keller, J.B.:
Geometrical theory of diffraction,
J. Opt. Soc. Am. \textbf{52} (1962), 116--130.

\bibitem{Kress99}
Kress, R.:
\emph{Linear Integral Equations}
(Springer, New York, 1999).

\bibitem{Lebedev97}
Lebedev, V.I.:
\emph{An Introduction to Functional Analysis in Computational Mathematics}
(Birkh\"{a}user, Boston, 1997).

\bibitem{NamburuEtAl04}
Namburu, R.R., Mark, E.R., Clarke, J.A.: 
Scalable electromagnetic simulation environment,
CMES Comput. Model. Eng. Sci. \textbf{5} (2004), 443--453. 


\bibitem{Saad03}
Saad, Y.:
\emph{Iterative Methods for Sparse Linear Systems}
(SIAM, 2003).

\bibitem{TongChew10}
Tong, M.S., Chew, W.C.:
Multilevel fast multipole acceleration in the Nystr\"om discretization of surface electromagnetic integral equations for composite objects,
IEEE Trans. Antennas and Propagation \textbf{58}(10) (2010), 3411--3416.
\end{thebibliography}
\end{document}